\documentclass{article}
\usepackage[utf8]{inputenc}
\usepackage[algonl,boxed,norelsize,lined]{algorithm2e}
\usepackage{amscd,amsmath,amsthm,amssymb,amsfonts,bbm}
\usepackage{mathtools}
\usepackage{arydshln}
\usepackage{fullpage}
\usepackage{comment}
\usepackage[all]{xy}
\usepackage{datetime}
\usepackage{color,mathtools}
\usepackage[english]{babel}
\usepackage[T1]{fontenc}
\usepackage{hyperref}
\usepackage{appendix}
\usepackage{enumitem}
\usepackage{theoremref,url}
\usepackage[new]{old-arrows}
\hypersetup{
     colorlinks = true,
     linkcolor = blue,
     anchorcolor = blue,
     citecolor = teal,
     filecolor = blue,
     urlcolor = black
     }

\usepackage{tocloft}
\usepackage{enumitem}

\usepackage{booktabs}

\usepackage[a4paper,top=3cm,bottom=2cm,left=3cm,right=3cm,marginparwidth=1.75cm]{geometry}

\usepackage{amsmath}
\usepackage{amssymb}
\usepackage{amsthm}
\usepackage{verbatim}

\usepackage{tikz,tikz-cd}
\usetikzlibrary{decorations.pathreplacing}
\usetikzlibrary{calc,cd}

\newcommand{\KK}{\mathbb{K}}
\newcommand{\CC}{\mathbb{C}}

\newcommand{\PP}{\mathbb{P}}

\newcommand{\mC}{\mathcal{C}}
\newcommand{\mI}{\mathcal{I}}
\newcommand{\mY}{\mathcal{Y}}
\newcommand{\mL}{\mathcal{L}}
\newcommand{\mM}{\mathcal{M}}
\newcommand{\mP}{\mathcal{P}}

\def\frk{\frak}               

\def\Phi{{\frk n}}
\def\Phi{{\frk N}}
\def\MC{{\mathcal C}}

\def\MX{{\mathcal X}}
\def\MY{{\mathcal Y}}

\newcommand{\ID}{I_{ \Delta}}

\newcommand{\plex}{\prec_{\rm lex}}

\newcommand{\conv}{\textrm{Conv}}
\renewcommand{\hom}{\textrm{Hom}}

\theoremstyle{definition}
\newtheorem{definition}{Definition}[section]
\newtheorem{remark}[definition]{Remark}

\theoremstyle{plain}
\newtheorem{theorem}{Theorem}
\newtheorem*{Theorem*}{Theorem}
\newtheorem*{Corollary*}{Corollary}
\newtheorem{lemma}[definition]{Lemma}
\newtheorem{prop}[definition]{Proposition}
\newtheorem{corollary}[theorem]{Corollary}
\newtheorem{example}{Example}[section]
\newtheorem{conjecture}[definition]{Conjecture}
\newtheorem*{question}{Question}

\theoremstyle{remark}

\title{Conditional probabilities via line arrangements\\ and point configurations}

\author{
 Oliver Clarke, Fatemeh Mohammadi and Harshit J Motwani}

\newcommand\Perp{\protect\mathpalette{\protect\independenT}{\perp}}
\def\independenT#1#2{\mathrel{\rlap{$#1#2$}\mkern2mu{#1#2}}}
\newcommand{\ind}[3][]{\left.#2 \Perp\!_{#1}\, #3 \inD}
\newcommand{\inD}[1][\relax]{\def\argone{#1}\def\temprelax{\relax}
  \ifx\argone\temprelax\right.\else\,\middle|#1\right.{}\fi}

\begin{document}

\maketitle

\noindent{\bf Abstract.} We study the connection between probability distributions satisfying certain conditional independence (CI) constraints, and point and line arrangements in incidence geometry. To a family of CI statements, we associate a polynomial ideal whose algebraic invariants are encoded in a hypergraph. The primary decompositions of these ideals give a characterisation of the distributions satisfying the original CI statements. Classically, these ideals are generated by 2-minors of a matrix of variables, however, in the presence of hidden variables, they contain higher degree minors. This leads to the study of the structure of determinantal hypergraph ideals whose decompositions can be understood in terms of point and line configurations in the projective space. 

\setcounter{tocdepth}{1}
\setlength\cftbeforesecskip{0.2pt}
{\hypersetup{linkcolor=black}\tableofcontents}

\section{Introduction}

This paper develops new links between conditional independence (CI) statements with hidden variables in statistics \cite{Studeny05:Probabilistic_CI_structures,DrtonSturmfelsSullivant09:Algebraic_Statistics,Sullivant}, incidence geometry \cite{richter2011perspectives,lee2013mnev} and the theory of Gr\"obner basis \cite{sturmfels1990grobner,Sturmfels02:Solving_polynomial_equations,herzog2011monomial}. In particular, we study a collection of polynomial ideals called \textit{hypergraph ideals}, 
that naturally arise from a collection of CI statements, and show that their irreducible components can be seen as spaces of point and line arrangements.  

\medskip

\noindent\textbf{Hypergraphs.} A hypergraph $H$ is a collection of subsets of $[n]:=\{1,\ldots,n\}$. We are mainly interested in algebraic properties of hypergraph ideals that encapsulate many families of ideals widely studied in the literature \cite{HSS,herzog2010binomial,ene2013determinantal,sidman2019geometric}. Let $X = (x_{i,j})$ be a $d \times n$ matrix of variables and $R = \KK[X]$ be the polynomial ring over an algebraically closed field $\KK$. We are mainly interested in the case when $\mathbb{K}= \mathbb{C}$. For subsets $A \subseteq [d]$ and $B \subseteq [n]$ of the same size, we write $[A|B]$ for the determinant of the submatrix of $X$ with rows indexed by $A$ and columns indexed by $B$. If $A = [d]$ then we omit it from the notation and simply write $[B]$ for the maximal minor on columns indexed by $B$. We say that an ideal of $R$ is determinantal if it is generated by determinants.
The \emph{hypergraph ideal} of $H$ is 
\[
I_H = \langle [A|B] : A \subseteq [d], \ B \in H, \ |A| = |B| \rangle \subseteq R.
\]
Hypergraph ideals arise in the study of conditional independence statements with hidden variables, see \S\ref{sec:application_CI_statements}. In particular, given a hypergraph $H,$ the main questions in this context are:
\begin{itemize}
    \item[(a)]\label{question:a} To determine whether the ideal $I_H$ is prime and if not what is its primary decomposition. 
\item[(b)] To find a Gr\"obner basis for the ideal $I_H$ and its prime components.
\end{itemize} 
Many known examples of determinantal ideals can be described as hypergraph ideals. For example, the classical ideals generated by all $t$-minors of $X$ can be considered as the hypergraph ideals of $H = \binom{[n]}{t}$ whose edges are all $t$-subsets of $[n]$. Every such ideal is prime \cite{hochster1971cohen}, and the collection of $t$-minors of $X$ forms a Gr\"obner basis for $I_H$ with respect to any term order that chooses the leading diagonal term as the initial term for each minor \cite{sturmfels1990grobner}. 
Other related families of ideals are the ideals of adjacent minors \cite{HSS} and generalised binomial ideals \cite{herzog2010binomial,Rauh}, that are associated to hypergraphs containing all intervals in $[n]$ of length $d$ and to graphs, respectively.
These ideals are all radical and their prime decompositions have a combinatorial description where each prime component is a hypergraph ideal admitting a square-free Gr\"obner basis.
The generalised binomial ideals arise in algebraic statistics as ideals of conditional independence statements in which all random variables are observed, see \cite{herzog2010binomial,Fink, Rauh,SwansonTaylor11:Minimial_Primes_of_CI_Ideals,Fatemeh2}. 

Here, we study a family of hypergraph ideals associated to conditional independence statements with {\em hidden variables}.
More precisely, let $n = k \ell$ for some positive integers $k$ and $\ell$. Define $\MY$ to be the following $k \times \ell$ matrix along with its rows $R_i$ and columns $C_j$ for each $i \in [k]$ and $j \in [\ell]$.
\[
\resizebox{\textwidth}{!}{
$\MY = (\MY_{i,j}) = 
\begin{bmatrix}
1      & k+1    & \dots & (\ell - 1)k + 1 \\
2      & k+2    & \dots & (\ell - 1)k + 2 \\
\vdots & \vdots &       & \vdots \\
k      & 2k     & \dots & \ell k \\
\end{bmatrix}, \ 
\begin{tabular}{lr}
    $R_i = \{\MY_{i,1}, \MY_{i,2}, \dots, \MY_{i,\ell}\}$ & its rows, \\
    $C_j = \{\MY_{1,j}, \MY_{2,j}, \dots, \MY_{k,j}\}$ & its columns.
\end{tabular}
$
}
\]
We define the hypergraph $\Delta$ on the ground set $[k\ell]$ as
\begin{eqnarray}\label{eq:Delta}
\Delta = \Delta(k, \ell) = 
\left\{
\binom{R_i}{3}, \binom{C_j}{2}:
i \in [k], j \in [\ell]
\right\}.
\end{eqnarray}
For small cases, see Examples~\ref{example:k2l5s2td3} and \ref{example:k3l4s2td3}.
Let $d\geq 3$ be an integer and $X = (x_{i,j})$ be a $d \times k\ell$ matrix of variables. We study the hypergraph ideal $I_\Delta$ in the polynomial ring $R = \KK[X]$. We show that the radical of $I_\Delta$ can be decomposed into ideals of line arrangements. Moreover, for small values of $k$, we prove that this yields a prime decomposition and in the case $k = 2$ we give an explicit combinatorial description of these prime components. This, in particular, solves a problem posed in \cite[\S4]{clarke2020conditional}. These ideals are the CI ideals associated to a conditional independence model described in Section~\ref{sec:application_CI_statements}. See also Remark~\ref{rem:CI}.

We would like to remark that the explicit computation of the primary decomposition of $I_{\Delta}$ quickly becomes intractable in computer algebra systems. For instance, the largest cases that can be calculated in \texttt{Macaulay2} \cite{M2} are $(k,\ell)=(2,5)$ and $(3,4)$. Examples~\ref{example:k2l5s2td3} and ~\ref{example:k3l4s2td3} show that:
\begin{itemize}
    \item For $(k,\ell)=(2,5)$ the ideal $I_\Delta$ has 171 prime components divided into 6 isomorphism classes.  
    \item For $(k,\ell)=(3,4)$ the ideal $I_\Delta$ has 319 prime components divided into 9 isomorphism classes.
\end{itemize}
Moreover, we show that for both examples above, 
the isomorphism classes of prime components correspond to certain point and line arrangements derived from $\Delta$.

\medskip

\noindent\textbf{Point and line arrangements.}\label{page:point} We will be mainly following the notation of \cite{richter2011perspectives}. An element of the affine variety of $I_\Delta$ is a $d \times n$ matrix that can be viewed as a collection of $n$ points in $\CC^d$. The generators of $I_\Delta$ guarantee that these points satisfy certain linearity conditions. For example, if three non-zero vectors in $\CC^d$ lie in a $2$-dimensional subspace, then we say that their corresponding points lie on a common line in the projective space $\CC\PP^{d-1}$. This happens if and only if all $3$-minors of the corresponding $d\times 3$ submatrix vanish.
Since $I_\Delta$ is generated by $2$ and $3$-minors, it is natural to use line arrangements to decompose its associated variety. A line arrangement is an abstract combinatorial tool which allows us to parametrise collections of points in $\CC\PP^{d-1}$ by specifying which points lie on the same projective line. 
A configuration of a line arrangement is a collection of points in $\CC^d$ that minimally satisfy the requirements of the line arrangement. We will investigate the line arrangements that appear in $I_\Delta$ and show that, for certain cases, the set of configurations is an irreducible affine variety. 

More precisely, an \textit{incidence structure} is a triple $(\mP,\mL,\mI)$ where $\mP$ is the set of points, $\mL$ is the set of lines and $\mI \subseteq \mP \times \mL$ is the set of incidences. If $(p, \ell) \in \mI$ then we say $p$ lies on $\ell$ or $\ell$ passes through $p$. A well-studied example of an incidence structure is the projective plane \cite{richter2011perspectives}. A \textit{point and line arrangement on $[n]$}, or simply a line arrangement, is an incidence structure such that: the set of points $\mP$ is a partition of $[n]$, each line contains at least three distinct points and any pair of points lies on at most one common line. By abuse of notation, we call each element of $[n]$ a point of the line arrangement. If $i, j \in [n]$ belong to the same point $P \in \mP$ then we say $i$ and $j$ coincide.
A line arrangement may also have a distinguished set of \emph{zero points}. Formally, $L$ is a line arrangement on $[n]$ with zero points $S \subset [n]$ if $L$ is a point and line configuration on $[n] \backslash S$.
We think of the non-zero points of a line arrangements as points in the projective space satisfying the condition that three points lie on a line in projective space if and only if they lie on a line in the arrangement.

Associated to each line arrangement $L$ is an ideal $I_L$, see Definition~\ref{def:line_arr_ideal}. For any ideal $I \subseteq \KK[x_1, \dots, x_n]$ its vanishing locus is $V(I) = \{x \in \KK^n : f(x) = 0, \text{ for all } x \in I \}$. The vanishing locus of the ideal $I_L$ can be understood as the Zariski closure of the collection of configurations.
A \emph{configuration of $L$} in $\KK^d$ is a collection of vectors $A = (a_1, \dots, a_n) \in (\KK^{d})^n$ such that a subset of vectors $\{a_{i_1}, \dots, a_{i_r}\}$ is linearly independent if and only if $B = \{i_1, \dots, i_r\}$ satisfies the following three properties: $B \cap S = \emptyset$, no pair of points in $B$ coincide and no three points in $B$ lie on the same line.
A configuration $A$ is often written as a matrix $A \in \KK^{d \times n}$ where the points $a_i$ are the columns of $A$.
The collection of all configurations is called the \textit{configuration space of $L$} and is denoted by $Z^o_L \subset \KK^{d \times n}$. Its Zariski closure is denoted $V_L = \overline{Z^o_L}$ and coincides with the variety of the line arrangement $V(I_L)$. The set $Z^o_L$ is a Zariski open subset of the vanishing locus of polynomials associated to the line arrangement, see Definition~\ref{def:line_arr_polys}.

\medskip

In this paper, we examine the algebraic properties of the hypergraph ideals $I_\Delta$  through the lens of incidence geometry. We study the structure of $I_\Delta$ in terms of certain point and line arrangements. The first main result of this paper is the following, see Theorems~\ref{thm:intersection_thm}, \ref{thm:k=2_min_prime_decomp} and Corollary~\ref{cor:irred_4_or_less_lines}.

\begin{theorem}\label{thm:intro_intersection}
Let $k \ge 2, \ell \ge 3$ and $d \ge 3$. The variety $V(I_\Delta)$ can be decomposed into the configuration spaces of a family of line arrangements. Moreover if $k = 2$ then $I_\Delta$ is radical and its minimal prime components are hypergraph ideals which are classified combinatorially.
\end{theorem}
We prove the second part of this theorem in a number of distinct steps. We use Gr\"obner bases to show that the hypergraph ideals appearing in the decomposition of $I_\Delta$ are radical. Perturbation arguments show that these ideals are also prime. In particular, we use the fact that the Zariski topology is coarser than the Euclidean topology in $\CC^n$, as proved in Lemma~\ref{prop:topology}.

In cases when it is not possible to find a complete combinatorial description of the prime components of $\sqrt{I_\Delta}$, we give a procedure for finding line arrangements whose configuration space is irreducible.
Our second main result is the following theorem which provides an inductive process to prove the irreducibility of the variety of a line arrangement. Here, for any line arrangement $L$ and any line $\ell$ of $L$, the line arrangement obtained from $L$ by removing $\ell$ has all points of $L$ which do not lie solely on $\ell$ and has all lines of $L$ except $\ell$ with the same incidence structure.  

\begin{theorem}\label{thm:irred_line_arr_build_up}
Let $L$ be a line arrangement and $\ell$ be a line of $L$. Let $L'$ be the line arrangement obtained from $L$ by removing $\ell$ and assume that $\ell$ intersects $L'$ in at most two points. If the configuration space of $L'$ is irreducible then so is the configuration space of $L$. 
\end{theorem}

\noindent\textbf{Outline of paper.} In Section~\ref{sec:main}, we study the hypergraph ideal $I_\Delta$ for all $k \ge 2$ and $\ell \ge 3$. In Section~\ref{sec:comb_line_arr}, we associate a line arrangement $Ar_\Delta(S)$ to each $I_\Delta$ and prove the Decomposition Theorem (Theorem~\ref{thm:intersection_thm}). 
In Section~\ref{sec:small_line_arr}, we categorise line arrangements into types and enumerate those with at most four lines. We then use geometric techniques in Section~\ref{sec:geom_pf_primeness} to prove that the corresponding ideals of these arrangements are all irreducible, see Corollary~\ref{cor:irred_4_or_less_lines}.
In Section~\ref{sec:k=2_case}, we give a complete description of $I_\Delta$ when $k = 2$. In Section~\ref{sec:k=2_I_S_gens}, we define the line arrangement $L(S)$ associated to $I_S$. We show that $\sqrt{I_S}$ coincides with the ideal of the line arrangement $L(S)$ using a perturbation argument. We also show that the ideals $\sqrt{I_S}$ decompose $\sqrt{I_\Delta}$ into prime components. In Section~\ref{sec:k=2_gb}, we show that the canonical generating set for $I_S$ forms a Gr\"obner basis and $I_S$ is radical. In Section~\ref{sec:k=2_min_comp}, we determine the minimal prime components of $\sqrt{I_\Delta}$ and prove Theorem~\ref{thm:k=2_min_prime_decomp}.
In Section~\ref{sec:application_CI_statements}, we consider the applications to algebraic statistics. More precisely, we show that our main result can be seen as a hidden variable analogue to the intersection axiom, see Section~\ref{sec:CI_intersection}.
In Section~\ref{sec:further_qus}, we provide further computations and questions arising from our work. In particular, we consider Gr\"obner bases of the ideals of line arrangements with three lines.

\section{Irreducible decomposition of \texorpdfstring{$V(I_\Delta)$}{VID}}\label{sec:main}
Let $\Delta$ be the hypergraph defined in \eqref{eq:Delta}. In this section, we explore the irreducible decomposition of the associated variety of $I_\Delta$ denoted by $V(I_{\Delta})$. The starting point is to prove the Decomposition Theorem, namely Theorem~\ref{thm:intersection_thm}. 
We show that each component in this decomposition corresponds to a certain line arrangement. Even though the associated variety of a line arrangement is not in general irreducible, we show that this is true for line arrangements with at most four lines.

\subsection{Line arrangements from \texorpdfstring{$\Delta$}{D}}\label{sec:comb_line_arr}
Our main tool for understanding the prime components of the ideal $\sqrt{I_{\Delta}}$ 
are line arrangements associated to $\Delta$, which we collect into families $Ar_\Delta(S)$ for each subset $S \subset [k\ell]$.
Each line arrangement $L \in Ar_\Delta(S)$ gives rise to a radical ideal $I_L$. We show that the intersection of all such ideals $I_L$ is equal to $\sqrt{I_{\Delta}}$. For particular families of line arrangements, we show that $I_L$ is prime. However, arbitrary line arrangements may not correspond to 
prime ideals, see \cite[Example~3.8]{MatroidsCIStatements}. 

\begin{definition}[Compatible line arrangements]\label{def:comb_line_arr}
Recall the  $k \times \ell$ matrix on integers $\mY$ whose rows and columns are $R_i$ and $C_j$ respectively for $i \in [k]$ and $j \in [\ell]$.
Fix a subset $S$ of $ [k\ell]$. A line arrangement $L = (\mP, \mL, \mI)$ on $[k\ell]$ with zero set $S$ is \textit{compatible with $\Delta$} if the points in each column $C_i$ coincide in $L$ and the points in each row lie on a line. Recall our notation from Page~\pageref{page:point} that $\mP$ is a collection of subsets which partitions $[n]$. Compatible line arrangements satisfy the following:
\begin{itemize}
    \item For each column $C_i$ of $\mY$, there exists $P \in \mP$ such that $(C_i\backslash S) \subseteq P$,
    \item For each row $R_i$ of $\mY$, if $R_i$ contains at least three non-zero distinct points then all points in $R_i$ lie on a line in $\mL$.
\end{itemize}
We denote by $Ar_\Delta(S)$ the collection of all compatible line arrangements with zero set $S$. 
\end{definition}
To simplify our notation, we may denote each set $\{i_1,\ldots, i_t\}$
by $i_1\cdots i_t$, where $i_1<\cdots<i_t$. 

\begin{example}{\rm
Suppose that $k = 2$ and $\ell = 4$. The rows and columns of $\mY$ are
$
R_1 = 1357, R_2 = 2468,
C_1 = 12, 
C_2 = 34, 
C_3 = 56$
and $C_4 = 78$. Let $S = 358$ be a subset. We construct four different line arrangements $L_1, L_2, L_3$ and $L_4$ and check whether they are compatible with $\Delta$.
\begin{itemize}
    \item Let $\{12, 4, 6, 7\}$ be the distinct non-zero points of $L_1$ and define a single line containing all the points. In this case $L_1$ is compatible with $\Delta$ since $L_1$ contains a line passing through the three non-zero distinct points $2,4,6$ in $R_2$.

    \item Let $\{124, 67 \}$ be the distinct non-zero points of $L_2$ which contains no lines. Note that $2$ and $4$ coincide and so $R_2$ contains only two distinct points. Therefore $L_2$ is compatible with $\Delta$.

    \item Suppose that $L_3$ is a line arrangement with distinct non-zero points $\{12, 4, 6, 7 \}$. Suppose that there is a line which passes through $2, 4$ and $7$ but not $6$. Then $L_3$ is not compatible with $\Delta$ since $R_2$ contains three distinct points that do not lie on a line.
    
    \item Suppose that $L_4$ is a line arrangement with distinct non-zero points $\{1, 24, 67 \}$. Then $L_4$ is not compatible with $\Delta$ since $1,2$ belong to the same column of $\mY$ but do not coincide in $L_4$.
\end{itemize}
}
\end{example}

For each $ d \ge 3$, we define the variety of a line configuration on $[n]$, which is a subset of $\KK^{d\times n}$.

\begin{definition}[Polynomials of a line arrangement]\label{def:line_arr_polys}
Fix $d$ and let $L = (\mP, \mL, \mI)$ be a line arrangement on $[n]$ with a collection of zero points $S$. Let $X$ be a $d \times n$ matrix of variables. Recall that we work over the polynomial ring $R = \KK[X]$. We define the following collection of polynomials
\begin{align*}
F(L) &= \bigcup_{j \in S} \{x_{i,j} : i \in [d] \} \\
     & \cup \bigcup_{P \in \mP}\left\{ [A | B] : A \subset [d], B \subset P, |A| = |B| = 2 \right\} \\
     & \cup \bigcup_{B}\left\{ [A | B] : A \subset [d], |A| = 3 \right\} \\
     & \cup \bigcup_{\ell_1, \ell_2} \{ [A|B] : A \subseteq [d], B \subseteq \{p \in P \subseteq [n] : (P, \ell_1 ) \in \mI \textrm{ or } (P, \ell_2) \in \mI\}, |A| = |B| = 4 \}
\end{align*}
where the third union is taken over triples $B = \{i,j,k \}$ of distinct non-zero points in $L$ that lie on a common line and the final union is taken over all pairs of lines $\ell_1, \ell_2$ which intersect at a point in $L$. Note that if $d = 3$, then $F(L)$ contains no $4$-minors. We define the ideal $N_L = \langle F(L) \rangle$.
\end{definition}

\begin{definition}\label{def:line_arr_ideal}
Fix $d \ge 3$ and let $L = (\mP, \mL, \mI)$ be a line arrangement on $[n]$ with zero points $S$. 
A basis $B \subseteq [n]$ is a maximal subset of size at most $d$ such that: $B \cap S = \emptyset$, no two points in $B$ coincide and no three points in $B$ lie on the same line.
We call the hypergraph ideal $I_{\{B\}}$ a basis ideal of $L$ and write it as $I_B$.
We define the space of configurations of $L$ to be the Zariski open subset $Z_L^o = V(N_L) \backslash \cup V(I_{B})$  where the union is taken over all bases of $L$.
The variety of the line arrangement is the Zariski closure $V_L = \overline{Z_L^o}$ and its associated ideal is $I_L = I(Z_L^o)$.
\end{definition}

\begin{remark}[Vanishing set and ideal of a line arrangement]\label{rmk:IL_containment}
\begin{itemize}
    \item[(i)] Let $L$ be a line arrangement on $[n]$ with zero points $S$. Let $J = \bigcap I_B$ be the intersection of all basis ideals of $L$. Since the variety of a line arrangement is the set difference between $V(N_L)$ and $V(J)$, the ideal of the line arrangement from 
Definition~\ref{def:line_arr_ideal} is the radical of the saturation $I_L = \sqrt{(N_L : J^\infty)}$.
\item[(ii)] By definition, for each line arrangement $L$ compatible with $\Delta$, the ideal $N_L$ contains $I_\Delta$. Since $I_L$ is obtained from $N_L$ by taking its radical and quotienting, it follows that $I_L \supseteq \sqrt{I_\Delta}$.
\item[(iii)] The polynomials in $F(L)$ do not necessarily generate the ideal $I_L$, see Example~\ref{example:three_lines}. But, for $k = 2$, they not only generate but also form a Gr\"obner basis for $I_L$, see Section~\ref{sec:k=2_case}. 
\end{itemize}
\end{remark}

Each configuration of a line arrangement contains enough data to recover the line arrangement. We make this construction explicit as follows.

\begin{definition}[Line arrangement from a matrix]
Let $A \in \KK^{d \times n}$ be a matrix with columns $\{a_1, \dots, a_n\} \subseteq \KK^d$. We define a line arrangement $\mM_A = (\mP, \mL, \mI)$ from $A$ as follows. Let $S = \{i : a_i = \underline 0\}$ be the collection of zero columns of $A$. We consider the non-zero columns of $A$ to be points in projective space $\PP^{d-1}$. We define $\mP = \{P_i\}_i$ to be exactly these projective points. That is, $a,b \in [n] \backslash S$ coincide in $\mM_A$ if and only if $\underline{a} = \lambda \underline{b}$ for some $\lambda \in \KK\backslash\{0\}$.

We define the lines $\mL = \{L_i\}_i$ to be the maximal collections of three or more points which lie on a line in the projective space with the natural incidences $\mI$. Explicitly, let $i,j,k \in [n]$ be three distinct points. Then $a_i, a_j, a_k$ are linearly dependent if and only if $i,j,k$ lie on a common line in $\mL$ which happens if and only if all $3$-minors of the submatrix $A_{ijk}$ vanish.
We denote \emph{the point and line arrangement of $A$} by $\mM_A$. Note that, by construction, we have $A \in Z^o_{\mM_A}$.
\end{definition}

Let us illustrate this construction with a simple example.

\begin{example}{\rm
Consider the matrix
\[
A = 
\begin{bmatrix}
0 & 1 & 2 & 0 & 1 & 0 \\
0 & 0 & 0 & 1 & 1 & 0 \\
0 & 0 & 0 & 0 & 0 & 1 \\
\end{bmatrix}
\in \CC^{3 \times 6}.
\]
The corresponding point and line arrangement $\mM_A = (\mP, \mL, \mI)$ has zero points $S = \{1 \}$ and four other distinct points $\mP = \{23, 4, 5, 6 \}$. There is a unique line in this arrangement which passes through the points $23, 4$ and $5$. Calculating the ideal of this line arrangement gives
\[
I_{\mM_A} = \langle x_{1,1}, x_{2,1}, x_{3,1}, [12|23], [13|23], [23|23], [245], [345] \rangle.
\]
}
\end{example}

\noindent\textbf{Notation.} Let $A \in \KK^{d \times n}$ be a matrix and $S \subseteq [n]$ be a subset. We denote by $A_S$ the submatrix of $A$ containing the columns indexed by $S$. 

\smallskip

With the language of compatible line arrangements, we now prove that $V(I_\Delta)$ is the union of configuration spaces $Z^o_L$ where $L$ ranges over all line arrangements $Ar_\Delta(S)$ for all $S \subseteq [k\ell]$.

\begin{theorem}[Decomposition Theorem]\label{thm:intersection_thm}
Let $k \ge 2, \ell \ge 3$ and $d \ge 3$. Then,
\[
\sqrt{I_{\Delta}} = \bigcap_{
\substack{
L \in Ar_\Delta(S),\\
S \subset [k\ell]}
} I_L.
\]
\end{theorem}

\begin{proof}
By Remark~\ref{rmk:IL_containment}(ii), we have $\sqrt{I_{\Delta}} \subset I_L$ for all $S$ and $L \in Ar_\Delta(S)$ which proves one side of the inclusion. To prove the other side, by the Hilbert's Nullstellensatz it is enough to show that
\[
V(I_{\Delta}) \subset \bigcup_{
\substack{L \in Ar_\Delta(S),\\ S \subset [k\ell]}
} V(I_L).
\] 
Let $A \in V(I_{ \Delta})$ and suppose $S$ is the set of indices of the zero columns in $A$. We now show that the line arrangement $\mM_A$ is compatible and belongs to $Ar_\Delta(S)$.
To see this, take any column $C_i$ of $\mY$. Since $A \in V(I_{ \Delta})$ we have that all $2$-minors of $A_{C_i}$ vanish. Hence each non-zero column of $A_{C_i}$ defines the same point in $\PP^{d-1}$. Hence $C_i \backslash S$ coincide in $\mM_A$. Next, take any row $R_i$ of $\mY$. Since $A \in V(I_{ \Delta})$, we have that all $3$-minors of $A_{R_i}$ vanish. Hence, it follows that all non-zero columns of $A_{R_i}$ lie on the same line in $\PP^{d-1}$. Therefore, if there are at least three distinct points in $R_i \backslash S$, then they are contained in a line of $\mM_A$.
So $\mM_A$ is a compatible with $\Delta$ hence $\mM_A \in Ar_\Delta(S)$. Furthermore, by construction of the ideal $I_{\mM_A}$, we have $A \in Z^o_{\mM_A} \subseteq V(I_{\mM_A})$.
\end{proof}

\begin{remark}
The proof of Theorem~\ref{thm:intersection_thm} shows that the variety $V(I_\Delta)$ is in fact the union of configurations spaces $Z^o_L$ of line arrangements compatible with $\Delta$.
\end{remark}

\begin{example}[Three lines meeting at a point]\label{example:three_lines}
{\rm Let $L$ be the line arrangement on $\{1, \dots, 7\}$ with no zero points. The points are $\mP = \{P_1, P_2, \dots, P_7 \}$ where $P_i = \{i \}$ is a singleton and the lines are $\mL = \{\ell_1, \ell_2, \ell_3 \}$. For the incidences, the points $P_1, P_2, P_3$ lie on $\ell_1$, $P_1, P_4, P_5$ lie on $\ell_2$ and $P_1, P_6, P_7$ lie on $\ell_3$. 
This is the line arrangement number $6$ in Figure~\ref{fig:3xn_line_arrangements}, which depicts each type of line arrangement with at most three lines. Fix $d = 3$. A generating set for the associated ideal is:
\[
I_L = 
\langle 
[123], [145], [167],
[234][567] - [235][467]
\rangle.
\]
Let us consider $I_\Delta$ when $k = 3$ and $\ell = 7$. By Theorem~\ref{thm:intersection_thm}, a line arrangement with the same incidence structure to $L$ appears among the irreducible components of $I_\Delta$. We note that in this case $[234][567]-[235][467]$ lies in $I_L$ but is not contained in the ideal generated by $F(L)$. 
For line arrangements with at most three lines, the only non-determinantal generators of the line arrangement ideals are precisely the condition satisfied by three lines meeting at a point.
}
\begin{figure}
    \centering
    \resizebox{\textwidth}{!}{
    \includegraphics[scale=0.50]{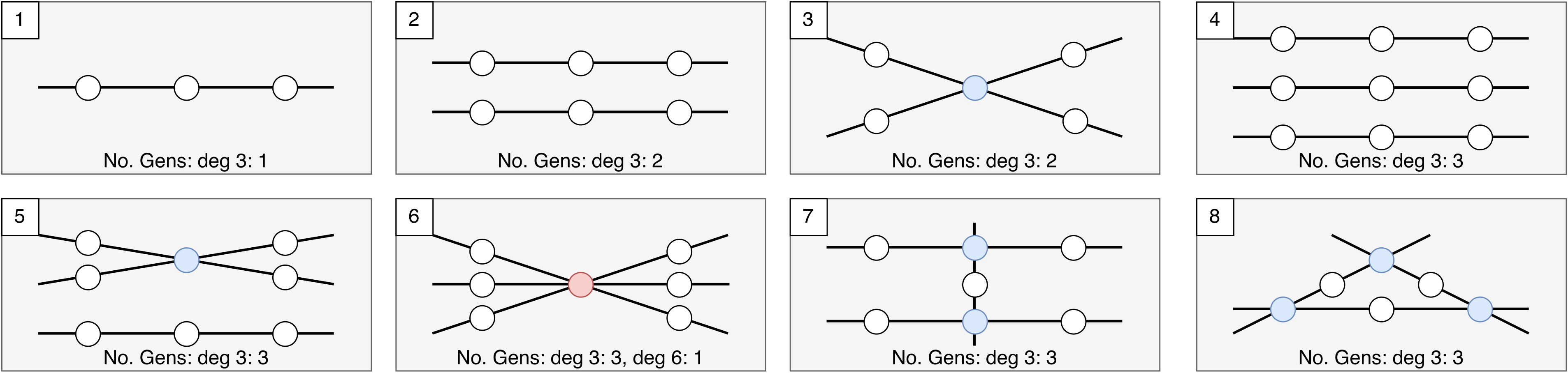}
    }
    \caption{Diagrams for all line arrangements with at most $3$ lines. Note that all possible line arrangements are obtained by adding points and labels at non-intersection points. In the diagrams, the intersection points are marked with colours; blue for double, and red for triple intersections. Below each line arrangement $L$, we list the number of generators of each degree for $I_L$. For example, in number $6$ a generating set of $I_L$ consists of $3$ polynomials of degree $3$ and $1$ of degree $6$.}
    \label{fig:3xn_line_arrangements}
\end{figure}
\end{example}
\subsection{Small line arrangements}\label{sec:small_line_arr}
 
To show that Theorem~\ref{thm:intersection_thm} yields a prime decomposition of $\sqrt{I_{\Delta }}$, it suffices to show that for each line arrangement $L$ appearing in the decomposition, the ideal $I_L$ is prime. Note that determining the primeness of $I_L$ is challenging, as its generating set is not known.

In the following example we will consider line arrangements that contain at most four lines. The defining characteristic of a line arrangement is its incidence structure. We partition the set of line arrangements into classes based on their incidences. We say that two line arrangements have the same incidence structure if there is a bijection between the lines that preserves incidences.

\begin{example}{\rm
In Figures~\ref{fig:3xn_line_arrangements} and \ref{fig:4xn_line_arrangements} we illustrate the incidence structures of line arrangement with at most four lines. For each class we illustrate the smallest line arrangement with the given incidence structure and summarise information about a minimal generating set of the corresponding ideal. We note that all line arrangements appearing in this example lie in the same class as a line arrangement compatible with $\Delta$ for some $k$ and $\ell$.
\begin{itemize}
    \item[$k=2$:] There are two classes of line arrangements with two lines determined by whether or not the lines intersect. In Section~\ref{sec:k=2_case}, we show that the polynomials $F(L)$ form  Gr\"obner bases for their respective ideals. These ideals are generated by $2,3$ and $4$-minors of $X$. In particular, if the two lines intersect then the ideals contain $4$-minors, since for any configuration $A$ of $L$ the columns corresponding to points on these lines lie in a $3$-dimensional linear subspace.

    \item[$k=3$:] There are five classes of arrangements with three lines, see Figure~\ref{fig:3xn_line_arrangements}. The only line arrangement that is not generated by minors is number 6 when three lines meet at a point.
    
    \item[$k=4$:] There are sixteen types of line arrangements with four lines, see Figure~\ref{fig:4xn_line_arrangements}.
    For all line arrangements except numbers $10,14,15$ and $16$, we have computationally verified that $I_L$ are prime.
    Through our calculations we observe that the only non-determinantal polynomials appearing in these ideals are the constraints for $3$ lines to meet at a point, see Example~\ref{example:three_lines}.
\end{itemize} 
}
\begin{figure}
    \centering
    \resizebox{\textwidth}{!}{
    \includegraphics[scale=0.5]{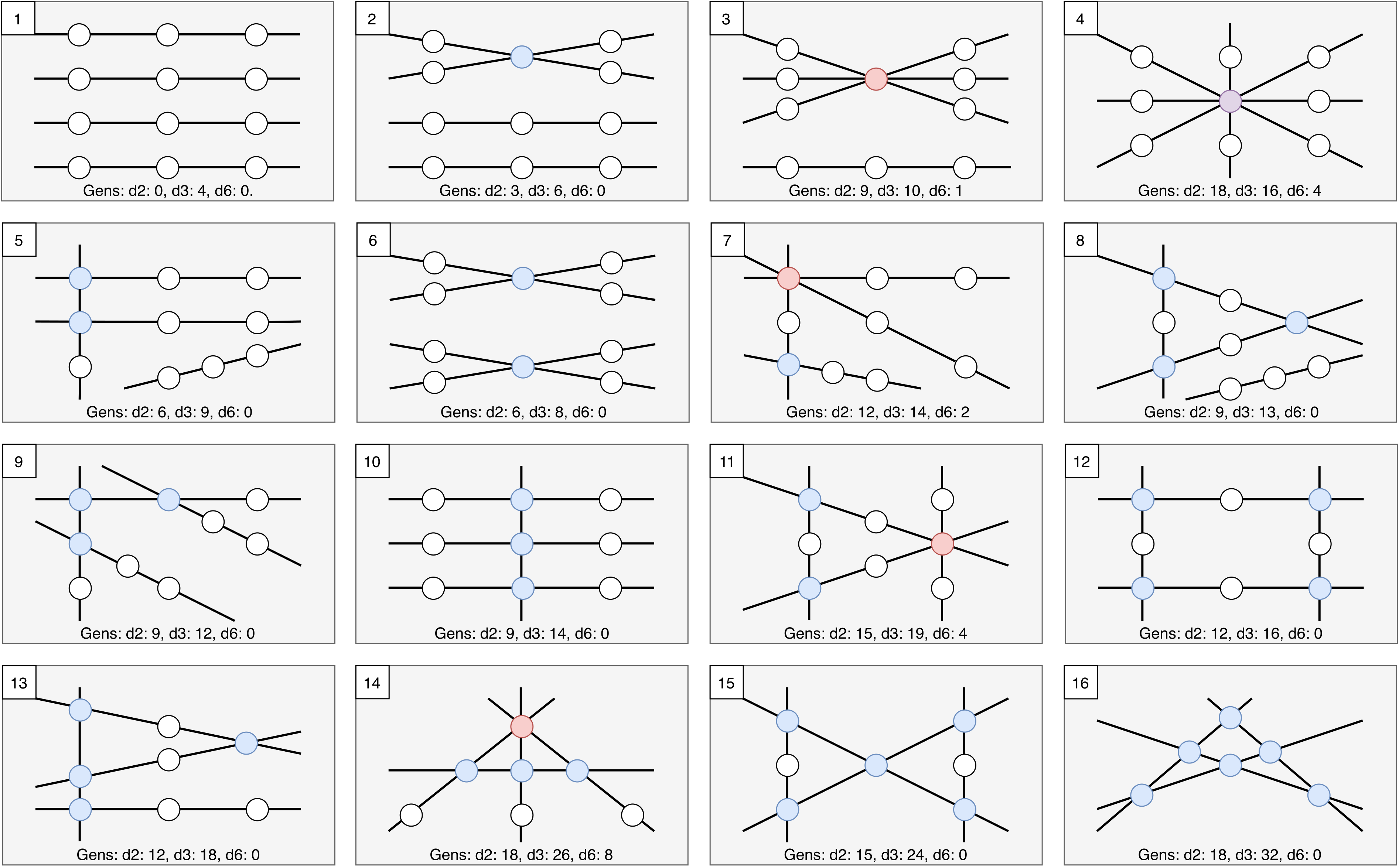}
    }
    \caption{Diagrams for line arrangements of $4$ lines. The line arrangements are ordered by the number of intersection points that are shaded.  Note that all other line arrangements on $4$ lines are obtained by adding and removing non-intersection points to the given lines.
    }
    \label{fig:4xn_line_arrangements}
\end{figure}
\end{example}

\subsection{Geometric proof of primeness}\label{sec:geom_pf_primeness}

We now give a geometric proof of Theorem~\ref{thm:irred_line_arr_build_up} and show that the varieties $V_L$ are irreducible where $L$ is a line arrangement with at most four lines, see Figure~\ref{fig:4xn_line_arrangements}. We fix $d = 3$ throughout this section for simplicity, so all point and line configurations live in the projective plane. However the same arguments easily extend to all $d \ge 3$.
Throughout this section we will use $\hom_{\CC}(\CC^a, \CC^b)$ to denote the set of $b \times a$ matrices with entries in $\CC$. However, as the notation suggests, we think of a point $\phi \in \hom_{\CC}(\CC^a, \CC^b)$ as a linear map $\phi: \CC^a \rightarrow \CC^b$.

\medskip

We first make precise the notion of removing a line or point from a line arrangement as follows.

\begin{definition}
Let $L = (\mP, \mL, \mI)$ be a line arrangement on $[n]$. For a line $\ell \in \mL$ we define the line arrangement $L \backslash \ell = (\mP', \mL \backslash \ell, \mI \cap (\mP' \times (\mL \backslash \ell))$ to be the line arrangement where $\mP'$ is the collection of points of $L$ which do not lie solely on $\ell$.
For a point $p \in [n]$ we define $L \backslash p = (\mP', \mL', \mI \cap (\mP' \times \mL') )$ to be the line arrangement where $\mL' \subseteq \mL$ is the collection of lines which contain at least three points which are not $p$ and $\mP'$ is the collection of non-empty sets $P \backslash p$ for all $P \in \mP$.
\end{definition}

\begin{remark}\label{prop:fiber_irred}
Throughout the paper, we will repeatedly use the following simple topological facts along with a result from algebraic geometry (assuming that $\KK$ is an algebraically closed field):
\begin{itemize}
    \item The continuous image of an irreducible topological space is irreducible \cite[\href{https://stacks.math.columbia.edu/tag/0379}{Lemma 0379}]{stacks-project}.
    \item Any non-empty open subset of an irreducible space is irreducible \cite[Example~1.1.3]{hartshorne2013algebraic}.
    \item A product of irreducible varieties over $\KK$ is irreducible 
    \cite[Exercise~10.1.E]{Vakil}.
\end{itemize}
\end{remark}

We begin with a very simple example that illustrates the structure of the proofs and the application of the results listed in Remark~\ref{prop:fiber_irred} that will follow.
\begin{example}\label{example:3_pts_on_a_line} 
{\rm Let $L$ be the line arrangement with three points lying on a line. Explicitly the points are $\mP = \{P_1, P_2, P_3 \}$ and each point is incident to the unique line of $L$. Let us show that the set $Z_L^o$ is irreducible. Consider the subset
\[
X = \{(x, \phi, (y_1, y_2)) : \phi(1,0) = x \} \subset \CC^3 \times \hom_\CC(\CC^2, \CC^3) \times (\CC^2)^2
\]
which is isomorphic to $\CC^3 \times \CC^3 \times (\CC^2)^2 $. In particular, $X$ is irreducible, since it is a product of irreducible varieties over an algebraically closed field.
Next we consider the map $\psi: X \rightarrow \CC^{3 \times 3}$, which is defined by $\psi(x, \phi, (y_1, y_2)) = (x, \phi(y_1), \phi(y_2)) \in \CC^{3 \times 3}$. This notation means $(x, \phi(y_1), \phi(y_2))$ is a $3 \times 3$ matrix whose columns are $x, \phi(y_1)$ and $\phi(y_2)$. As $\psi$ is a morphism, it is a continuous map and so the image of $\psi $ is irreducible in $\CC^{3 \times 3}$. 
Moreover, we see that $Z_L^o$ is an open subset of the image of $\psi$ as follows. Any point in the image of $\psi$ is a collection of three vectors in $\CC^3$ lying in the column space of $\phi$. If we remove from $\psi(X)$ the degenerate cases in which any pair of these columns are equal or any column is equal to the origin, then we obtain an open subset of $\psi(X)$. This open subset consists of all non-degenerate configurations of a line passing through three points in the projective plane, in other words the open subset is $Z_L^o$. Therefore, $Z_L^o$ is irreducible.
}
\end{example}

Throughout the proofs in this section, an identical argument to the above example shows that $Z_L^o$ is an open subset of the image of $\psi$.
For example, by changing $(\CC^2)^2$ to $(\CC^2)^{n-1}$ in Example~\ref{example:3_pts_on_a_line} we can prove that the space of configurations of $n$ points on a line is irreducible.

\begin{proof}[{\bf Proof of Theorem~\ref{thm:irred_line_arr_build_up}}]
Let $m$ be the dimension of the ambient space containing $V_L$ and recall that $Z^o_{L \backslash \ell}$ is the open subset of $V_{L \backslash \ell}$ consisting of all configurations of $L\backslash \ell$. We proceed by constructing a map $\psi : X \rightarrow \CC^m$ from an irreducible space $X$ to $\CC^m$ such that the image $\psi(X)$ contains $Z^o_L$ as an open subset. We construct $X$ and $\psi$ by taking cases on the number of points of incidence between $\ell$ and the configuration $L \backslash \ell$. By assumption there are either zero, one or two such incidences.
\medskip

\textbf{Case 1.} Assume there are no incidences. Define $X = Z^o_{L \backslash \ell} \times \hom(\CC^2, \CC^3) \times (\CC^2)^{|\ell|}$, which is irreducible as it is a product of irreducible spaces. Let us define the map 
\[
\psi: X \rightarrow \CC^m : \ 
(x, \phi, (y_1, \dots, y_{|\ell|})) \mapsto (x, \phi(y_1), \dots, \phi(y_{|\ell|})).
\] We see that $Z^o_L$ is an open subset of the image of $\psi$. Hence $V_L = \overline{Z^o_L}$ is irreducible.

\smallskip

\textbf{Case 2.} Assume that there is exactly one point of incidence between $\ell$ and $L \backslash \ell$. Given a configuration $x \in Z^o_{L \backslash \ell}$ let us denote by $x_1 \in \CC^3$ the point of incidence, i.e.~$x_1$ is a column of $x$ corresponding to the point of incidence between $\ell$ and $L \backslash \ell$. We define
\[
X = \{(x, \phi, (y_1, \dots, y_{|\ell|-1})) : \phi(1,0) = x_1 \} \subset Z^o_{L \backslash \ell} \times \hom_\CC(\CC^2, \CC^3) \times (\CC^2)^{|\ell|-1}.
\]
Notice that for each $x \in Z^o_{L \backslash \ell}$, the space of homomorphisms $\phi : \CC^2 \rightarrow \CC^3$ such that $\phi(1,0) = x_1$ can be seen as the space of $3 \times 2$ matrices whose first column is $x_1$, i.e.~it is isomorphic to $\CC^3$. Therefore $X$ is isomorphic to $Z^o_{L \backslash \ell} \times \CC^3 \times (\CC^2)^{|\ell|-1}$. Hence $X$ is irreducible.
Let us define the map 
\[
\psi: X \rightarrow \CC^m : \ 
(x, \phi, (y_1, \dots, y_{|\ell|-1})) \mapsto (x, \phi(y_1), \dots, \phi(y_{|\ell|-1})).
\] We see that $Z^o_L$ is an open subset of the image of $\psi$. Hence $V_L = \overline{Z^o_L}$ is irreducible.

\smallskip

\textbf{Case 3.} Assume that there are exactly two points of incidence between $\ell$ and $L \backslash \ell$. Given a point $x \in Z^o_{L \backslash \ell}$ let us denote by $x_1, x_2 \in \CC^3$ columns of $x$ which are representatives for the points of incidence. We define
\[
X = \{ (x, \phi, (y_1, \dots, y_{|\ell|-2})) : \phi(1,0) = x_1, \ \phi(0,1) = x_2 \} \subset
Z^o_{L \backslash \ell} \times \hom_\CC(\CC^2, \CC^3) \times (\CC^2)^{|\ell|-2}.
\]
For each $x \in Z^o_{L \backslash \ell}$ we see that $\phi$ is totally determined by $x_1$ and $x_2$. Hence $X$ is isomorphic to $Z^o_{L \backslash \ell} \times (\CC^2)^{|\ell|-2}$ which is irreducible. Let us define the map 
\[
\psi: X \rightarrow \CC^m : \ 
(x, \phi, (y_1, \dots, y_{|\ell|-2})) \mapsto (x, \phi(y_1), \dots, \phi(y_{|\ell|-2})).
\] We see that $Z^o_L$ is an open subset of the image of $\psi$. Hence $V_L = \overline{Z^o_L}$ is irreducible.
\end{proof}

\begin{corollary}\label{cor:irred_4_or_less_lines}
The ideals of line arrangements with at most $4$ lines are irreducible.
\end{corollary}

\begin{proof}
By Theorem~\ref{thm:irred_line_arr_build_up} we can build up most line arrangements with at most $4$ lines from smaller line arrangements. Figure~\ref{fig:build_up_4_line_configs} shows how to build up connected configurations, i.e.~arrangements in which there is a path between any two points along lines within the configuration. In the disconnected cases, one can build up the configurations by looking at each connected component separately. The only type of line arrangement that cannot be built up is line arrangement $16$ in Figure~\ref{fig:4xn_line_arrangements} that contains $4$ lines such that each pair of lines intersects at a point. Let $L$ be a line arrangement of this type with $m$ points and label the lines in this arrangement $\ell_1, \ell_2, \ell_3$ and $\ell_4$. Then $L \backslash \ell_1$ is a line arrangement of type $8$ in Figure~\ref{fig:3xn_line_arrangements} and $V_{L \backslash \ell_1}$ is irreducible by Theorem~\ref{thm:irred_line_arr_build_up}.

\begin{figure}
    \centering
    \resizebox{0.75\textwidth}{!}{
    \includegraphics{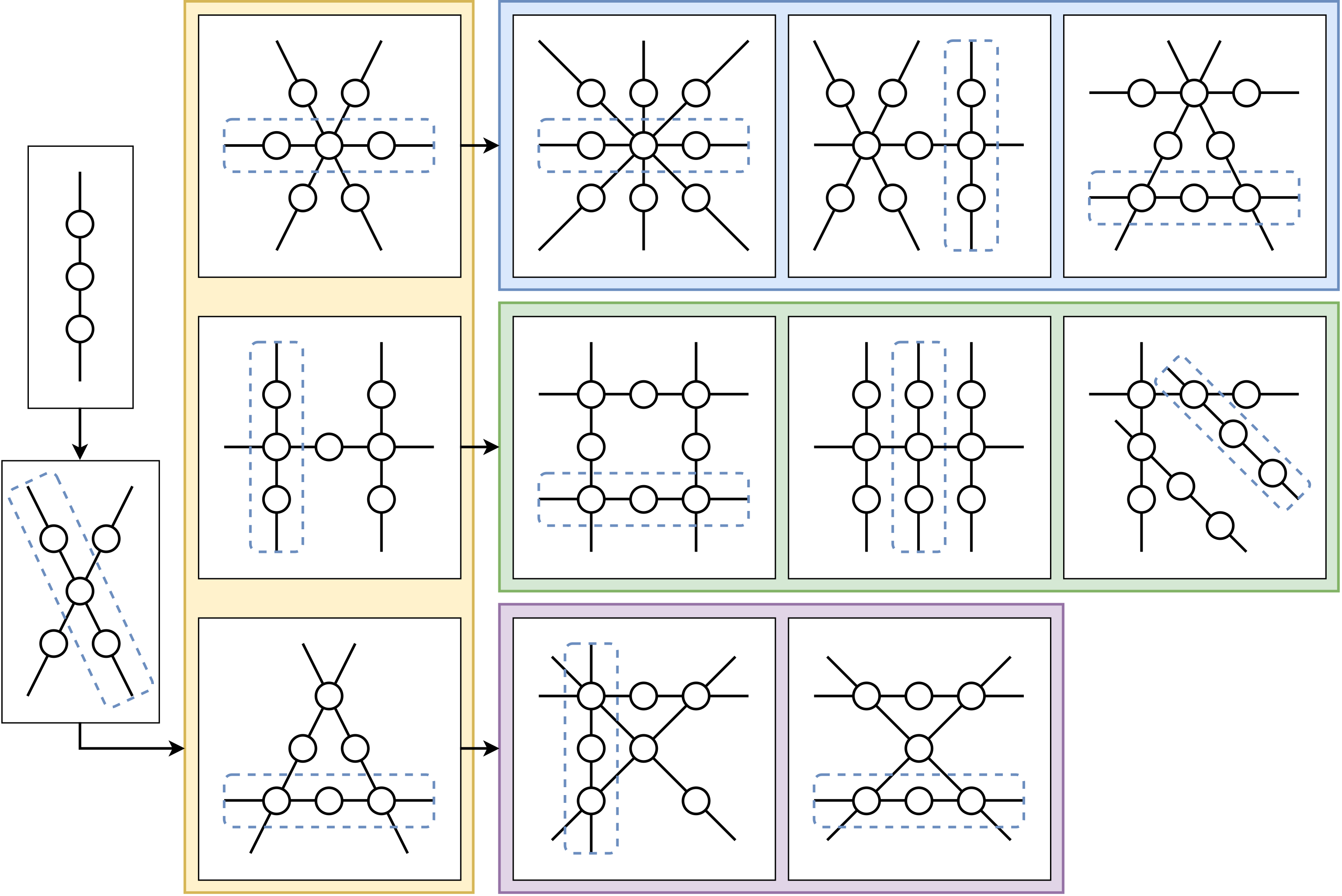}}
    \caption{Depiction of the process of building up point and line configurations with $4$ lines. The base case is a single line shown in the top left. At each stage, the newly added line is indicated with a dotted box. The large shaded boxes indicate collections of configurations which are built up from a common arrangement. 
    }
    \label{fig:build_up_4_line_configs}
\end{figure}

Recall that $Z^o_L \subset V_L$ is the open subset of configurations in which all lines and points in the arrangement are distinct. Denote by $P_{12}$ the point in the line arrangement which is the intersection of $\ell_1$ and $\ell_2$. For each configuration $x \in Z^o_L$ we denote by $x_{12}$ a non-zero vector in $\CC^3$ which spans the $1$-dimensional linear subspace containing all columns of $x$ corresponding to $P_{12}$.
Similarly we define $x_{13}, P_{13}, x_{14}$ and $P_{14}$ for the intersection points $\ell_1, \ell_3$ and $\ell_1, \ell_4$ respectively. By Theorem~\ref{thm:irred_line_arr_build_up}, the variety $V_{L \backslash \{\ell_1, P_{13}, P_{14}\}}$ is irreducible. We denote $W^o \subset V_{L \backslash \{\ell_1, P_{13}, P_{14}\}}$ for the open subset in which all projective lines and points are distinct in the projective plane.
We define
\[
X = \{(x, \phi, (y_1, \dots, y_{|\ell_1|-3})) : \phi(1,0) = x_{12} \} \subset W^o \times \hom_\CC(\CC^2, \CC^3) \times (\CC^2)^{|\ell_1|-3}.
\]
We have that $X$ is isomorphic to $W^o \times \CC^3 \times (\CC^2)^{|\ell_1|-3}$ and so $X$ is irreducible. 
Given a point $x \in X$, we identify the lines $\ell_1, \dots, \ell_4$ of $L$ with the corresponding $2$-dimensional linear subspaces of $\CC^3$ spanned by the columns of $x$ containing the points on each respective line. We now define 
\[
X' = \{(x, z_{13}, z_{14}) : z_{13} \in \ell_1 \cap \ell_3, \ z_{14} \in \ell_1 \cap \ell_4 \} \subset X \times \CC^3 \times \CC^3.
\]
We have that for each point $x \in X$, the $2$-dimensional linear subspaces  $\ell_1$ and $\ell_3$ are distinct. Since the subspaces are contained within a $3$-dimensional subspace, we have that the intersection $\ell_1 \cap \ell_3$ is a $1$-dimensional linear subspace. In other words, the intersection of two projective lines is a projective point.
 
In particular 
$z_{13}$ can be free chosen on this subspace and so is parametrised by $\CC$. Similarly for the subspaces $\ell_1$ and $\ell_4$. And so $X'$ is isomorphic to $X \times \CC \times \CC$ hence $X$ is irreducible.
Recall that $m$ is the number of points in the line arrangement $L$, we define the map
\[
\psi : X' \rightarrow \CC^{3m} : ((x, \phi, (y_1, \dots, y_{|\ell_1|-3})), z_{13}, z_{14}) \mapsto (x,\phi(y_1), \dots, \phi(y_{|\ell_1|-3}), z_{13},z_{14})).
\]
The image of $\psi$ contains $Z^o_L$ as an open subset. Since $X'$ is irreducible, $Z^o_L$ is irreducible as well.
\end{proof}

\begin{remark}
We note that even though all results of this section are stated for point and line arrangements contained within the projective plane, they can be carefully extended to higher dimensional projective spaces. We would like to remark that
 irreducibility of point and line arrangements is also studied in \cite{MatroidsCIStatements} for a so-called \emph{tree-like} configurations. These are line arrangements in which the underlying graph is a tree. The paper \cite{MatroidsCIStatements} takes a matroidal perspective and relies on an understanding of \emph{combinatorial closures} of matroid ideals. These ideals can be thought of as generalisations of determinantal ideals $\langle F(L) \rangle$. However, their associated varieties are not irreducible, even for matroids on very small ground sets. 
\end{remark}

\section{The two row case}\label{sec:k=2_case}
In this section, we tackle the challenging problem of finding generating sets for the ideals of point and line arrangements. In \cite{sidman2019geometric}, the authors generalise the non-determinantal polynomial appearing in the ideal of the line arrangement in Example~\ref{example:three_lines}. In particular, they construct non-trivial polynomials of arbitrarily high degree for various line arrangements. Furthermore, by the Mn\"ev and Sturmfels universality theorem we know that the variety of point and line configurations can have arbitrary singularities up to a mild equivalence, see e.g.~\cite{sturmfels1989matroid, vakil2006murphy, lee2013mnev}.

Note that the line arrangements that appear in Theorem~\ref{thm:intersection_thm} contain at most $k$ lines and $\ell$ points. Therefore the decomposition may be refined to a minimal prime decomposition of $I_\Delta$ for all $\ell$ and $k \le 4$. We proceed to investigate the prime ideals of this decomposition in the case $k = 2$. 
More precisely, we prove the second part of Theorem~\ref{thm:intro_intersection}, namely Theorem~\ref{thm:k=2_min_prime_decomp} using Gr\"obner bases and perturbation arguments. This result, in particular, provides an explicit minimal prime decomposition of $\sqrt{I_{\Delta }}$. Moreover, we show that the prime components can be all represented as hypergraph ideals that coincide with the ideals of line arrangements described in Definition~\ref{def:line_arr_polys}.

\subsection{Hypergraph constructions and ideals}
We begin by describing precisely properties of hypergraphs that we will use to construct generating sets of the prime components of $I_\Delta$ when $k = 2$.

\begin{definition}
Let $H$ be a hypergraph on the vertex set $[n]=\{1,\ldots,n\}$. Recall that $H$ is identified with its edge set. We define the following notions for hypergraphs.
\begin{itemize}
    \item A \emph{path} from vertex $a$ to vertex $b$, is a sequence of edges $e_1, \dots, e_r \in H$ such that $a \in e_1$, $b \in e_r$ and for each $i \in [r-1]$ we have that $e_i \cap e_{i+1} \neq \emptyset$.
The \emph{connected component} of $H$ containing a vertex $v$ is the set of vertices $w$ such that there exists a path from $v$ to $w$.
    \item An \emph{$r$-clique} of $H$ is a hypergraph on a subset $S \subseteq [n]$ such that all $r$-subsets of $S$ are edges of $H$.
 The \emph{$r$-completion} of $H$ denoted by $H^{[r]}$ is the smallest hypergraph containing $H$ and satisfying the property that if $a_1, \dots, a_r$ are vertices lying in the same connected component of $H$ then $\{a_1, \dots, a_r\}$ is an edge in $H^{[r]}$. Equivalently $H^{[r]}$ is the hypergraph obtained from $H$ by adding all $r$-subsets which lie in a single connected component.
    \item The \emph{induced hypergraph $H_{V}$} of $H$, for a subset $V \subseteq [n]$, is the hypergraph obtained from $H$ by removing all edges  containing a vertex in $V$. Explicitly, $H_V = \{ e \in H : e \cap V = \emptyset \}$.
    Unless otherwise stated, the hypergraph $H_V$ is assumed to be defined on the vertex set $[n] \backslash V$.
\end{itemize}
\end{definition}

The hypergraph $\Delta(k, \ell)$ is connected, i.e.~it has only one connected component. In the case $k = 2$, it is straightforward to show that any induced hypergraph of $\Delta$ contains at most four connected components. For each subset $S \subseteq [k\ell]$, we define the hypergraph ideal $I_S$ as follows. 

\begin{definition}\label{def:I_S_and_H(S)}
Let $D(k, \ell)$, or simply $D$ if $k$ and $\ell$ are fixed, be the collection of subsets $S \subseteq [k \ell]$  such that either $S = \emptyset$ or the induced hypergraph $\Delta_S$ has at least two components. We denote by $D^c$ be the complement of $D$.
We denote $\mC(S)$ for the union of all $C_i$, $i \in [\ell]$ with $C_i \cap S = \emptyset$. 
\begin{itemize}
    \item If $S \in D^c$ then we define
    \[
    H(S) = 
    \left\{ 
    S,
    \binom{\mC(S)}{2},
    \binom{(R_1 \backslash S) \cup \mC(S)}{3},
    \binom{(R_2 \backslash S) \cup \mC(S)}{3},
    \binom{[k\ell]\backslash S}{4}
    \right\}.
    \]
    
    \item If $S\in D$ then we define $H(S) = (\Delta_S)^{[3]} \cup S$, where $(\Delta_S)^{[3]}$ is the $3$-completion of $\Delta_S$.
\end{itemize}
In the hypergraph $H(S)$, the set $S$ is thought of as a collection of singletons. We define the ideal $I_S$ to be the hypergraph ideal $I_{H(S)}$. We usually write $I_0$ for $I_{\emptyset}$, which is defined as \begin{equation*}\label{eq:L0}
    I_0=I_{\emptyset}=\langle [B] : B \subset [\ell k], \ |B| = t \rangle + I_\Delta.
\end{equation*}
\end{definition}

\begin{remark}
In the above definition when $S \in D^c$ we note that the hypergraph ideal contains some $4$-subsets. We can think of these $4$-subsets as a condition for all columns of a point $A \in V(I_H)$ to lie in a $3$-dimensional subspace of $\CC^d$. Furthermore, we can think of the columns $(R_1 \backslash S \cup \MC(S))$ and $(R_2 \backslash S \cup \MC(S))$ as lying on two (distinct) projective lines, i.e.~$2$-dimensional linear subspaces. These two projective lines meet at a point corresponding to $\MC(S)$.
\end{remark}

\subsection{Generators of prime components}\label{sec:k=2_I_S_gens}

Hypergraph ideals admit a canonical minimal generating set that is straightforward to compute. More precisely, an edge $e \in H$ gives rise to redundant generators if and only if there exists an edge $e' \in H$ such that $e' \subsetneq e$.
For the ideals $I_S$ with $S \in D^c$, this means that all $4$-minors intersecting the columns in $\mC(S)$ are redundant. We illustrate the minimal generating set of $I_S$ in an example.

\begin{example}[Example of $I_S$]\label{example:I_star_S}
{\rm Let $k=2,\ell=6$ and
$S = \{1,3,6,8\}$. So in this case we have
\[
\mY = 
\begin{bmatrix}
\boxed{1} & \boxed{3} & 5 & 7 & 9 & 11 \\
2 & 4 & \boxed{6} & \boxed{8} & 10 & 12
\end{bmatrix}
\]
where the boxed elements are those in $S$. The induced hypergraph $\Delta_S$ is given by
\[
\Delta_S = \left\{ 
\{9,10\},
\{11,12\},
\binom{\{5,7,9,11\}}{3}, \binom{\{2,4,10,12\}}{3}
\right\}.
\]
Note that $\Delta_S$ has exactly one connected component so $S \in D^c$. The columns of $\mY$ that are properly contained in $[k\ell] \backslash S$ are the right most two columns and so $\mC(S) = \{9,10,11,12 \}$. Thus we have:
\[
    H(S) = \left\{
    1,3,6,8,
    \binom{\{9,10,11,12 \}}{2}, 
    \binom{\{5,7,9,10 ,11,12\}}{3},
    \binom{\{2,4,9,10,11,12 \}}{3},
    \binom{[12] \backslash \{1,3,6,8 \}}{4}
    \right\}.
\]
In this case, the canonical minimal generating set of $I_S$ contains only the $4$-minors corresponding to the edge $\{2,4,5,7\}$. All other $4$-minors are generated by the $2$ and $3$-subsets in $H(S)$.
}
\end{example}

The ideals $I_S$ are in fact ideals of line arrangements $L(S)$ defined as follows.

\begin{definition}\label{def:line_arr_L(S)}
Let $S \subseteq [k\ell]$ be a subset of points of the matrix $\mY$. Recall that $\mC(S)$ is the union of all columns $C_i$ of $\mY$ with $C_i \cap S = \emptyset$. We define the line arrangement $L(S) = (\mP, \mL, \mI)$ as follows.
\begin{itemize}
    \item If $S = \emptyset$ then we define the points $\mP$ to be the columns $C_i$ for all $i \in [\ell]$ and define $\mL$ to be a single line which passes through all points.
    \item If $S \neq \emptyset$ then we define the points $\mP$ to be $\mC(S)$ along with all singletons $C_i \backslash S$ for $i \in [\ell]$. Note, if $C_i \subseteq S$ for some $i \in [\ell]$ then this does not give rise to a point of $L$. There are (at most) two lines in $\mL$ which are incident to the points $\{P \in \mP : P \cap R_i \neq \emptyset \}$ for $i \in \{1,2 \}$ respectively. If any line would contain fewer than three points, then it is not a line of $L$.
\end{itemize}
\end{definition}

\begin{example}[Continuation of Example~\ref{example:I_star_S}]
{\rm Suppose $k = 2, \ell = 6$ and $S = \{1,4,6,8\}$. The line arrangement $L(S) = (\mP, \mL, \mI)$ has the points and lines
\[
\mP = \{
\{2\}, \{4 \}, \{5\}, \{7\},
\{9,10,11,12\}\} \quad\text{and}\quad
\mL = \{\ell_1, \ell_2\}.
\]
The points $\{5\}, \{7\}, \{9,10,11,12\}$ lie on $\ell_1$ and $\{2\}, \{4 \}, \{9,10,11,12\}$ lie on $\ell_2$. Notice that the point $\{9,10,11,12 \}$ is the union of two columns of $\mY$. Also the lines $\ell_1$ and $\ell_2$ contain the corresponding points in rows $R_1$ and $R_2$ of $\mY$ respectively. Hence, $L(S)$ is compatible with $\Delta$.
}
\end{example}

Recall from Definition~\ref{def:I_S_and_H(S)} that $I_S$ is a hypergraph ideal. We first state our main results.

\begin{theorem}\label{thm:radical_I_S_=_I_L}
Let $k = 2$ and $S\subset [k\ell]$. Then the radical of the hypergraph ideal $I_S$ is the ideal of the line arrangement $L(S)$ with zero points $S$, that is $\sqrt{I_S} =  I_{L(S)}$.
\end{theorem}

\begin{prop}\label{prop:I_S=Ar_D(S)}
Let $k = 2$. For each subset $S \subseteq [k\ell]$ we have
\[
\sqrt{I_S} = I_{L(S)} = \bigcap_{L \in Ar_\Delta(S)} I_L.
\]
\end{prop}

Proposition~\ref{prop:I_S=Ar_D(S)} together with Theorem~\ref{thm:intersection_thm} shows that the collection of ideals $\sqrt{I_S}$ is sufficient to decompose $\sqrt{I_\Delta}$. That is $\sqrt{I_\Delta} = \bigcap_S \sqrt{I_S}$ where the intersection is taken over all subsets $S \subseteq [k\ell]$. In subsequent sections we will show that $I_S$ is radical, hence $I_S$ coincides with the ideal $I_{L(S)}$.

\begin{remark}
For a given set $S \subseteq [k\ell]$, the line arrangement $L(S)$ contains two lines intersecting at a point if and only if $S \in D^c$, $|S \cap R_1| \ge 2$ and $|S \cap R_2| \ge 2$.
\end{remark}

The following results make use of a perturbation argument. In broad terms this means, given a set $Z \subseteq \CC^n$ and a point $x \in \CC^n$ outside of $Z$, if $x$ is a limit point of $Z$, with respect to the standard norm on $\CC^n$, then $x$ lies in the Zariski closure of $Z$. This follows from the fact that the Zariski topology is coarser than Euclidean metric topology over $\CC^n$.

\begin{lemma}\label{prop:topology}
Any closed subset $U \subseteq \CC^n$ in the Zariski topology is closed in the Euclidean topology.
\end{lemma}

\begin{proof}
Since $U$ is closed in the Zariski topology, $U = \bigcap_i f_i^{-1}(0)$ for some finite collection of polynomials $f_i$. Since polynomials are continuous in the Euclidean topology, we have $f_i^{-1}(0)$ is closed for each $i$ therefore $U$ is closed. 
\end{proof}

Before giving the proof of Theorem~\ref{thm:radical_I_S_=_I_L}, we consider the line arrangements consisting of one line. In order to apply the perturbation argument we work over $\CC$, which contrasts the rest of the paper where we work over an arbitrary algebraically closed field.

\begin{lemma}[Perturbation into a projective line] \label{lem:perturb_in_p_line}
Fix two natural numbers $n,d$ and assume that $\epsilon > 0$. Let L = $(\mP, \mL, \mI)$ be a line arrangement on $[n]$ that either has a single line containing all points or, if $|\mP| \le 2$ then, has no lines. Let $P \in \mP$ be a distinguished point. Let $A \in V(N_L)$ be a $d \times n$ matrix satisfying the polynomials $F(L)$. 
Then there exists $A' \in Z^o_L$ such that $||A - A'|| < \epsilon$. 
Furthermore, $A'$ can be chosen so that the columns of $A$ and $A'$ indexed by $P$, lie in a common $1$-dimensional linear subspace of $\CC^d$. 
\end{lemma}

\begin{proof}
\textbf{Single point case.}
We begin with the case that $L$ is a line arrangement with a single point $\mP = \{P_1 \}$ where $P_1 = [n]$. In this case the bases $B$ of $L$ are singleton subsets of $[n]$ and so the basis ideals $I_B$ are monomial ideals. Therefore a point $A$ of $V(N_L)$ lies in $Z^o_L$ if and only if all columns are non-zero.
Let $W \subseteq \CC^d$ be a $1$-dimensional linear subspace containing all columns of $A$. 
Let $N = |\{i \in [n] : A_i = \underline 0\}|$ be the number of zero columns of $A$. Let $\underline b \in W$ be a non-zero vector such that $||b|| < \epsilon / N$, that is the norm of the vector $b$ is bounded by $\epsilon / N$. We define the matrix $A' \in \CC^{d \times n}$ as follows. The column $i^{\rm th}$ of $A'$ is $A'_i = A_i$ if $A_i$ is non-zero, otherwise $A'_i = \underline b$. We have that all columns of $A'$ are non-zero and lie in $W$ so $A' \in Z^o_L$. Also we have that:
\[
||A - A'|| \le \sum_{A_i = \underline 0} ||\underline b|| = N||\underline b|| < \epsilon
\]
Therefore, the result holds in this case.

\medskip

\textbf{General case.} We now consider the general case, so let us write $\mP = \{P_1, \dots, P_r \}$ for the points of the line arrangement and assume that $r \ge 2$. Since there are no zero points in the line arrangement we have that $\mP$ is a partition of $[n]$. In this case the bases $B \subseteq [n]$ of $L$ are precisely $B = \{ i,j\}$ such that $i, j$ are distinct points of $L$. So a matrix $A \in V(N_L)$ lies in $Z^o_L$ if and only if for any basis $B = \{i,j \}$ the vectors $A_i$ and $A_j$ are linearly independent.
Given $A \in V(N_L)$, we have that all $3$-minors of $A$ vanish. So all columns of $A$ are contained in a $2$-dimensional linear subspace $W_2 \subseteq \CC^d$. For each point $P_i \in \mP$, we have that the columns $A_j$ for $j \in P_i$ lie in a $1$-dimensional linear subspace which we will denote $W_1(i)$ for $i \in [r]$. If all linear subspaces $W_1(i)$ are distinct, then by the ``single point case'', we can construct $A'$ so that all columns are non-zero. So it suffices to show that we can perturb $A$ to a matrix $A'$ whose corresponding linear subspaces $W_1(i)$ are all distinct. Suppose, without loss of generality, that $W_1(1) = W_1(2)$. 

We will now show how to perturb the point $P_2$ away from $P_1$.
Let $\underline{b_1} \in W_1(1)$ be a unit vector. Let $M = \max\{ ||A_i|| : i \in P_2 \}$ be the length of the longest vector of a point in $P_2$. 
We take $\underline{b_2} \in W_2$ such that the following conditions are satisfied 
\begin{itemize}
    \item $\underline{b_1}, \underline{b_2}$ is a basis for $W_2$,
    \item For all $i \in [r]$, the linear subspace $W_1(i)$ does not contain $\underline{b_1} + \underline{b_2}$,
    \item $||b_2|| < \epsilon / (M |P_2|)$.
\end{itemize}
Note that it is always possible to find such a vector $\underline{b_2}$ since there are only finitely many linear subspaces $W_1(i)$. 
We now define the $d \times n$ matrix $A'$ as follows. Each column of $A'_i$ indexed by $i \in [n] \backslash P_2$ is taken to be $A_i$. For each $i \in P_2$ we have that $A_i = \lambda_i \underline{b_1}$ for some $\lambda_i \in \CC$. We define $A'_i = \lambda_i(\underline{b_1} + \underline{b_2})$. As a result we have that all columns of $A$ indexed by $P_2$ lie in the $1$-dimensional subspace spanned by the vector $\underline{b_1} + \underline{b_2}$. By construction we have that $\underline{b_1} + \underline{b_2}$ does not lie in any of $W_1(i)$ for $i \in [r]$ and so we have moved the point $P_2$ away from all other points. We also have
\[
||A - A'|| \le 
\sum_{i \in P_2} ||A_i - A'_i|| =
\sum_{i \in P_2} |\lambda_i| \cdot ||\underline{b_2}|| \le
\sum_{i \in P_2} M \cdot ||\underline{b_2} || =
|P_2| \,  M \, ||\underline{b_2}|| < \epsilon.
\]
We apply this process inductively for every pair of points that coincide. Note that after applying this argument once, the linear subspace $W_1(1)$ is fixed. So if one of the two points that coincide is the distinguished point $P$ then we may take $P = P_1$ so that the linear subspace $W_1(1)$ is fixed.
\end{proof}

We now give a proof of Theorem~\ref{thm:radical_I_S_=_I_L} which is the first step to proving the correspondence between the line arrangements $L(S)$ and the ideals $I_S$. 

\begin{proof}[{\bf Proof of Theorem~\ref{thm:radical_I_S_=_I_L}}]
To simplify notation we write $L = L(S)$ for the point and line arrangement. It is easy to see that the generators of $I_S$ vanish on each point in the configuration space $Z^o_L$. And so $I_S \subseteq I_L$. For the converse, it suffices to show that $V(I_S) \subseteq V_L$.

\smallskip

Recall from the definition that $Z^o_L \subseteq V_L$ is an open subset. Take $A \in V(I_S) \backslash Z^o_L$. We will show that $A \in \overline{Z^o_L} = V_L$. To do this, we will show that for all $\epsilon > 0$, there exists $A' \in Z^o_L$ such that $||A - A'|| < \epsilon$, where 
\[
||(a_{i,j})|| = \sum_{i,j} |a_{i,j}|, \ (a_{i,j}) \in \CC^{d \times k\ell}
\]
is the usual norm. Fix $\epsilon > 0$. We consider all three possible cases for $S$ to construct $A'$.

\medskip

\textbf{Case 1.} Assume $S = \emptyset$. Since $L$ is a line arrangement with a single line containing all points, we may apply Lemma~\ref{lem:perturb_in_p_line} which gives the desired matrix $A'$. 
\medskip

\textbf{Case 2.} Assume $S \in D \backslash \emptyset$. Since $S \in D$, the lines in $L$ do not intersect. Hence, we can write $L = L_1 \cup \dots \cup L_m$ as a disjoint union of line arrangements that either have a single point or a single line containing all points. We may apply Lemma~\ref{lem:perturb_in_p_line} to each line arrangement $L_i$ to form $A'$. 

\medskip

\textbf{Case 3.} Assume $S \in D^c$.
The generators of $I_S$ include all $4$-minors of $X$. It follows that all columns of $A$ lie in a $3$-dimensional subspace $W_0 \subseteq \CC^d$. We have that the line arrangement $L = (\mP, \mL, \mI)$ has two lines $L_1, L_2$ that intersect at a point, call it $P$. In the matrix $A$, the columns corresponding to $L_1$ lie in a $2$-dimensional subspace, which we denote by $W_1$. Similarly, the columns corresponding to $L_2$ also lie in a $2$-dimensional subspace, which we denote by $W_2$. Since $W_1, W_2 \subseteq W_0$, by dimension counting we have that $\dim (W_1 \cap W_2) \ge 1$.

Suppose that $W_1 = W_2$. We will now describe how to perturb $W_2$ away from $W_1$. By Lemma~\ref{lem:perturb_in_p_line} we may assume that the columns corresponding to $P$ are non-zero.
Let $\underline{b_1} \in \CC^d$ be a unit vector whose span corresponds to the point $P$. Let $\underline{b_2} \in \CC^d$ be a vector such that $\underline{b_1}, \underline{b_2}$ is a basis for $W_1$. 
For each $i \in \{1, 2\}$, we write $\overline{L_i} \subseteq [n]$ for the points $i \in P_j$ incident to $L_i$. For all $i \in \overline{L_2} \backslash P$, we have $A_i = \lambda_i \underline{b_1} + \mu_i \underline{b_2}$ where $\mu_i \neq 0$. Let $M = \max\{ |\mu_i| : i \in P_j, \ P_j \neq P, \ (P_j, L_2) \in \mI\}$.
Let $\underline{b_3} \in \CC^d$ be any vector in $W_0$ not contained in $W_1$ such that $||\underline{b_3}|| < \epsilon / (M |\overline{L_2}|)$. We define the matrix $A'$ as follows. We define $A'_i = A_i$ for all $i \in \overline{L_1}$ and $A'_i = A_i + \mu_i \underline{b_3}$ for all $i \in \overline{L_2} \backslash P$.
By construction we have moved the plane containing all points in $L_2$ away from the plane containing all points in $L_1$. We also have that
\[
||A - A'|| \le
\sum_{i \in \overline{L_2} \backslash P} ||A_i - A'_i|| =
\sum_{i \in \overline{L_2} \backslash P} |\mu_i|\cdot||\underline{b_3}|| \le
\sum_{i \in \overline{L_2} \backslash P} M\cdot||\underline{b_3}|| \le
|\overline{L_2}|M\cdot||\underline{b_3}||
< \epsilon.
\]
By construction, the columns $\overline{L_1}$ and $\overline{L_2}$ of $A'$ lie in two distinct $2$-dimensional planes. Also, each point of the line arrangement lies in a distinct $1$-dimensional subspace of $W_0$ and so $A' \in Z^o_L$.
\end{proof}

We have seen that the ideals $\sqrt{I_S}$ are special examples of ideals of line arrangements. We now show that they are sufficient to decompose $\sqrt{I_\Delta}$.

\begin{proof}[{\bf Proof of Proposition~\ref{prop:I_S=Ar_D(S)}}]
Note that $L(S) \in Ar_\Delta(S)$. So we will show that $\sqrt{I_S} \subseteq I_L$ for each line arrangement $L \in Ar_\Delta(S)$.
Suppose that $S \in D$. Then by definition we have that the generators of $I_S$ are contained in $F(L)$ hence $I_S \subseteq I_L$. Since $I_L$ is radical, we have that $\sqrt{I_S} \subseteq I_L$.

Suppose $S \in D^c$. It suffices to show that $V_L \subseteq V(I_S)$. Take any point in $A \in Z^o_L$. Since the generators of $I_S$ contain $F(L)$, it remains to show that all $4$-minors of $A$ vanish. We take cases on the number of lines in $L$. Since $L \in D^c$, we have  $\mC(S) \neq \emptyset$ and $\mC$ is contained in a point of $L$. 

\medskip

\textbf{Case 1.} Assume $L$ contains no lines. We show that $L$ contains at most $3$ points. Suppose by contradiction that $L$ contains $4$ points, $P_1, P_2, P_3, P_4$. Without loss of generality assume $\mC(S) \subseteq P_1$. We must have that either $R_1$ or $R_2$ intersects at least two of the points $P_2, P_3, P_4$. Suppose $R_1$ intersects $P_2$ and $P_3$. It follows that $P_1,P_2,P_3$ lie on a line corresponding to $R_1$. Similarly for the other cases, we have that $L$ contains a line, a contradiction.
Since $L$ contains $3$ points, therefore all columns of $A$ lie in a $3$-dimensional linear space of $\CC^d$. Hence all $4$-minors of $A$ vanish.

\medskip

\textbf{Case 2.} Assume $L$ has exactly one line. We show that $L$ contains at most one point which does not lie on this line. Suppose there are two points $P_1$ and $P_2$ that do not lie on the line in $L$. Without loss of generality, assume that the line contains all points that intersect $R_1$. Suppose $\mC \subseteq P_3$ is a subset of a point of $L$ that lies on the line. Therefore $P_1$ and $P_2$ do not intersect $R_1$ and so they must intersect $R_2$. By definition of line arrangement, there must exist a line passing through $P_1, P_2, P_3$, corresponding to $R_2$, a contradiction.
Therefore $L$ contains a single line and at most one point off the line. And so all columns of $A$ lie within a $3$-dimensional linear subspace of $\CC^d$ and so all $4$-minors of $A$ vanish.

\medskip 

\textbf{Case 3.} Assume $L$ contains exactly two lines. Suppose $\mC \subseteq P_1$ is contained in a point of $L$. By assumption $A \in Z^o_L$. Therefore, all columns corresponding to point $P_1$ are non-zero. All columns of $L$ contained in a line lie in a $2$-dimensional linear subspace of $\CC^d$. Each column of $A$ indexed by $P_1$ is non-zero and lies in the intersection of these two $2$ linear subspaces. Therefore all columns of $A$ are contained in a $3$-dimensional linear subspace of $\CC^d$ and so all $4$-minors of $A$ vanish. 
\end{proof}

\begin{remark}
By Theorem~\ref{thm:intersection_thm} and Proposition~\ref{prop:I_S=Ar_D(S)} we have
\[
\sqrt{\ID} = \bigcap_{L \in Ar_\Delta(S), \ S\subseteq[k\ell]} I_L =  \bigcap_{S\subseteq[k\ell]} \sqrt{I_S}.
\]
Since the line arrangement $L(S)$ contains at most $2$ lines, by Corollary~\ref{cor:irred_4_or_less_lines}, we have that $I_{L(S)}$ is prime. Therefore the above intersection can be refined to a minimal prime decomposition.
\end{remark}

We now give an example of a prime decomposition of $\sqrt{I_\Delta}$. In this example we examine the largest $\ell$ that can be currently computationally verified. Due to the symmetries of the matrix $\MY$, many different subsets $S \subseteq [k\ell]$ give rise to isomorphic ideals $I_S$. So, to simplify the notation for the following example and subsequent sections, we group together similar subsets of $[k\ell]$ as follows.

\begin{definition}
Let $Sym(\mY) \le Sym([k\ell])$ denote the permutations that are induced by permutations of the rows and columns of $\mY$. We naturally have $Sym(\mY) \cong Sym([k]) \times Sym([\ell])$. Two ideals $I_S$ and $I_{S'}$ have \emph{the same type} if $S$ and $S'$ lie in the same $Sym(\mY)$ orbit.
\end{definition}

\noindent It is straightforward to show that $I_S$ and $I_{S'}$ are isomorphic if and only if they have the same type.

\begin{example}\label{example:k2l5s2td3}
{\rm Let $k = 2, \ell = 5$, so we have 
$\mY = 
\begin{bmatrix}
1   &3  &5  &7  &9 \\
2   &4  &6  &8  &10\\
\end{bmatrix}
$ and
\[
\Delta = \Delta(2, 5) = \left\{\{1,2 \},\{3,4 \},\{5,6 \},\{7,8 \},\{9,10 \},
\binom{\{1,3,5,7,9\}}{3}, \binom{\{2,4,6,8,10\}}{3} \right\}.
\]
The prime decomposition of $I_{\Delta}$ can be verified computationally in the case $d = 3$. We record the minimal prime components in the following table, up to isomorphism.
\begin{center}
\begin{tabular}{ccc}
    \toprule
    Type of Ideal  
    & Hypergraph of Ideal
    & \begin{tabular}{c}
        Number of such ideals \\
        up to isomorphism 
    \end{tabular} \\
    \midrule
    $I_0$               & 
    $\left\{C_i \textrm{ for } i \in [5], \ \binom{[10]}{3}\right\}$ & 
    1                     \\
    $I_{1,4,6,8}$       & 
    $\left\{1,4,6,8,\{9,10 \}, \binom{\{3,5,7,9,10\}}{3}\right\}$ &
    20                    \\
    $I_{1,3,6,8,10}$    &
    $\left\{1,3,6,8,10, \{5,7,9\} \right\}$ &
    10                    \\
    \midrule
    $I_{1,4}$         & 
    $\left\{1,4 \binom{\{5,6,7,8,9,10 \}}{2} \right\}$ &
    20                    \\
    $I_{1,3,6,8}$  & 
    $\left\{1,3,6,8,\{9,10 \}, \binom{\{5,7,9,10\}}{3}, \binom{\{2,4,9,10\}}{3}\right\}$ & 
    60                    \\
    $I_{1,4,6}$    & 
    $\left\{1,4,6, \binom{\{7,8,9,10\}}{2}, \binom{\{3,5,7,8,9,10\}}{3} \right\}$ &
    60                    \\
    \bottomrule
\end{tabular}
\end{center}
}
\end{example}

\begin{example}\label{example:delta_hypergraph}
{\rm In \cite{clarke2020conditional}, the authors consider the class of hypergraph ideals $I_{H}$ where
\[
H = \Delta'(k, \ell) = 
\left\{ 
\binom{R_i}{\ell}, \binom{C_j}{2} : i \in [k], \ j \in [\ell]
\right\}.
\]
Surprisingly, the prime components of $\sqrt{I_H}$ are hypergraph ideals that contain subsets of size one and two.
Moreover, the hypergraphs associated to the minimal prime components are uniquely identified by their singletons and so are denoted by $I_S$ for $S \subset [\ell k]$. It turns out that either $S = \emptyset$ or $\Delta_S$ is disconnected and contains exactly $\ell$ connected components corresponding to the columns of the matrix $\mY$.
In the case $\ell = 3$, the hypergraphs $\Delta$ and $\Delta'$ coincide and the ideals $I_S$ agree.
However, 
we have seen in Example~\ref{example:k2l5s2td3} that for $ \ell > 3$ there are prime components of the form $I_S$ where $\Delta_S$ is connected. Theorem~\ref{thm:k=2_min_prime_decomp} shows that if $k = 2$ then each component $I_S$ is a hypergraph ideal whose hypergraph can be uniquely identified by its singletons.
}
\end{example}

\subsection{Gr\"obner basis for  \texorpdfstring{$I_S$}{IS}}\label{sec:k=2_gb}

We will now show that the canonical generating set for $I_S$ forms a Gr\"obner basis with respect to a carefully chosen term order. We will use this to show that each ideal $I_S$ is, in fact, prime. 

\smallskip

First, for each isomorphism class of $I_S$ where $S \in D^c$, we choose a standard representative. Note that this is equivalent to choosing a representative hypergraph under the action of $Sym(\mY)$.

\begin{definition}[Hypergraph Representative]
Let $i, j , c$ be non-negative integers such that $i \le j$. We define $\mC(i,c) = \{x : i + 1 \le x \le i + c \}$ which is empty if $c = 0$. We define the hypergraph
\[
F(i,j,c) := 
\left\{
\binom{\mC(i,c)}{2},
\binom{\{1, \dots, i + c \}}{3}, 
\binom{\{i+1, \dots, i + c + j \}}{3},
\binom{[i + j + c] \backslash \mC(i,r)}{4}
\right\}.
\]
\end{definition}

\begin{prop}
For each $S \in D^c$, there exist unique natural numbers $i,j,c$ such that the hypergraph $H(S) \backslash S$ is isomorphic to $F(i,j,c)$.
\end{prop}

\begin{proof}
Given $S \in D^c$, let $\mC = \mC(S)$, $I = R_1 \backslash (S \cup \mC)$ and $J = R_2 \backslash (S \cup \mC)$. Then we have: 
\[
H(S) \backslash S = 
\left\{
\binom{\mC}{2},
\binom{I \cup \mC}{3},
\binom{J \cup \mC}{3},
\binom{[k\ell] \backslash (S \cup \mC(S))}{4}
\right\}.
\]
Let $i = |I|$, $j = |J|$ and $c = |\mC|$. We define an isomorphism of sets $\phi : [k\ell] \backslash S \rightarrow [i + j + c]$ such that $\phi(I) = \{1, \dots, i \}$, $\phi(\mC) = \{i+1, \dots, i+c \}$ and $\phi(J) = \{i+c+1, \dots, i+j+c\})$. 
Note that $I, J$ and $\mC$ are disjoint and so $\phi(H(S)\backslash S) = F(i,j,c)$.

To show uniqueness, suppose that $\phi: [i+j+c] \rightarrow [i'+j'+c']$ is a bijection that induces an isomorphism of hypergraphs between $F(i,j,c)$ and $F(i', j', c')$. Consider the edges of $F(i,j,c)$ of size two. These edges form a clique on $\mC(i,c)$, similarly for $F(i',j',c')$. So we must have that $\phi(\mC(i,c)) = \mC(i',c')$. Note that $|\mC| = c$ and $|\mC'| = c'$. Therefore $c = c'$.
We have that $F(i,j,c)$ contains two maximal cliques of $3$-subsets, which are on vertices $\{1, \dots, i+c \}$ and $\{ i+1, \dots, i+c+j\}$. These cliques have size $i+c$ and $j+c$ respectively. Similarly $F(i',j',c')$ has two maximal cliques of size $i' + c$ and $j' + c$. Since $\phi$ is an isomorphism it must send maximal cliques to maximal cliques. Since $i < j$ and  $i' < j'$ we deduce that $i = i'$ and $j = j'$.
\end{proof}

The difference between the ideals $I_S$ and $I_{F(i,j,c)}$ is that the latter does not contain any variable. If we show that the canonical generators of $I_{F(i,j,c)}$ form a Gr\"obner basis, with respect to some term order, then it will follow that the generators of $I_S$ form a Gr\"obner basis, with respect to any term order that extends the previous term order. This follows directly from Buchberger's algorithm since each variable in $I_S$ is relatively prime to the leading term of any other generator.
We denote $\plex$ for the lexicographic term order induced by the natural order of the variables:
\[ 
p_{u,i}>p_{v,m}
\quad\text{iff}\quad 
u<v,\text{\ or\ } u=v \ {\rm and\ } i<m.
\]

\begin{lemma}\label{lem:k=2_grobner}
For all non-negative integers $i \le j, c$ and subsets $S \in D$, the canonical minimal generating set of each ideal $I_{F(i,j,c)}$, $I_0$ and $I_S$ forms a Gr\"obner basis with respect to $\plex$.
\end{lemma}

\begin{proof}
First we consider $I_{F(i,j,c)}$. We proceed by Buchberger's criterion. Let $g_1, g_2$ be elements of the generating set for $I_{F(i,j,c)}$. By construction, the generators have degree $2,3$ or $4$ so we take cases on their degree. Additionally, we may assume that the initial terms of $g_1$ and $g_2$ are not relatively prime, otherwise it follows immediately that the $S$-polynomial reduces to zero. Let $P$ denote the smallest submatrix of $X$ which contains $g_1$ and $g_2$ as minors.

\medskip

\textbf{Case 1.} Let $\deg(g_1) = \deg(g_2) = 2$. The $2$-minors in the generating set of $I_{F(i,j,c)}$ are all $2$-minors from columns 
indexed by $\mC(i,c)$. By a classical result from  \cite{sturmfels1990grobner}, the collection of all $2$-minors of a matrix of variables forms a Gr\"obner basis for the ideal they generate with respect to $\plex$. So the $S$-polynomial $S(g_1, g_2)$ reduces to zero in the present case by the same reduction.

\medskip

\textbf{Case 2.} Let $\deg(g_1) = \deg(g_2) = 3$. Suppose that $g_1 = [c_1 c_2 c_3 \ | \ a_1 a_2 a_3]$ and $g_2 = [e_1 e_2 e_3 \ | \ b_1 b_2 b_3]$ where $1 \le c_1 < c_2 < c_3 \le d$ and $1 \le e_1 < e_2 < e_3 \le d$. By assumption $g_1$ and $g_2$ lie in a minimal generating set for $I_{F(i,j,c)}$, so we must have $|\{a_1, a_2, a_3\} \cap \mC(i,c) | \le 1$ and $|\{b_1,b_2,b_3\} \cap \mC(i,c)| \le 1$. Hence, without loss of generality, there are two cases:
\begin{itemize}
    \item $a_1, a_2, a_3, b_1, b_2, b_3$ all belong to $\{1, \dots, i+c \}$ or $\{i+1, \dots, i+j+c \}$,
    \item $a_1, a_2, a_3$ belongs to $\{1, \dots, i+c \}$ and $b_1,b_2,b_3$ belongs to $\{ i+1, \dots, i+j+c\}$. 
\end{itemize}

\smallskip

\textbf{Case 2.i.} Assume that $a_1, a_2, a_3, b_1, b_2, b_3$ all belong to $\{1, \dots, i+c \}$ or $\{ i+1, \dots, i+j+c \}$.
Note that $I_{F(i,j,c)}$ contains all  $3$-minors on columns $\{1, \dots, i+c \}$ and $\{i+1, \dots, i+j+c \}$. By a classical result, the collection of all $3$-minors form a Gr\"obner basis for the ideal they generate. Therefore the $S$-polynomial $S(g_1, g_2)$ reduces to zero by the same reduction as the classical case. 

\smallskip

\textbf{Case 2.ii.} Assume that $\{a_1, a_2, a_3 \} \subseteq \{1,\dots,i+c\}$ and $\{b_1,b_2,b_3\} \subseteq \{i+1,\dots,i+j+c\}$. Since the leading terms of $g_1$ and $g_2$ are not relatively prime, it follows that $a_3 = b_1 \in \mC(i,c)$ and $c_3 = e_1$. It follow that $P$ is a $5 \times 5$ matrix, and we may relabel its rows so that $g_1 = [123 | 123]$ and $g_2 = [345| 345]$. Let $J$ be the ideal generated by all the $3$-minors of $P$ with columns $\{1,2,3\}$ and $\{3,4,5\}$. We will show that the $S$-polynomial $S(g_1, g_2)$ reduces to zero inside $J$ with respect to $\plex$. By the division algorithm we find,
\begin{align*}
    S(g_1,g_2) &= x_{4,4}x_{5,5}g_1 - x_{1,1}x_{2,2}g_2 \\
    &= x_{4,5}x_{5,4}[123|123] + [35|45][124|123] - [25|45][134|123] + [15|45][234|123] \\
    &- [34|45][125|123] + [24|45][135|123] - [14|45][235|123] - [23|45][145|123] \\
    &+ [13|45][245|123] - [12|45][345|123] + [45|12][123|345] - [35|12][124|345] \\
    &+ [25|12][134|345] - [15|12][234|345] + [34|12][125|345] - [24|12][135|345] \\
    &+ [14|12][235|345] + [23|12][145|345] - [13|12][245|345] - x_{1,2}x_{2,1}[345|345].
\end{align*}

This shows that $S(g_1,g_2)$ reduces to zero in $J$. Since the same generators appear in $I_{F(i,j,c)}$, the $S$-polynomial also reduces to zero by the same reduction.

\medskip

\textbf{Case 3.} Let $\deg(g_1) = \deg(g_2) = 4$. Since the ideal $I_{F(i,j,c)}$ contains all $4$-minors of $X$, we have that the $S$-polynomial of $g_1$ and $g_2$ reduces to zero by a classical result.

\medskip

\textbf{Case 4.} Let $\deg(g_1) = 2 $ and $\deg(g_2) = 3$. We may assume that, of the columns defining $g_2$, at most one is taken from $\mC(i,c)$. Therefore, in $P$, the columns defining $g_1$ are adjacent.
We have that either the submatrix $P$ has $3$ rows or at least $4$ rows.

\textbf{Case 4.i.} Assume that $P$ has $3$ rows. We write
$
P = \begin{bmatrix}
x_1 & x_2 & x_3 & x_4 \\
y_1 & y_2 & y_3 & y_4 \\
z_1 & z_2 & z_3 & z_4 \\
\end{bmatrix}.
$
Note that there are two cases for $g_2 = [a_1a_2a_3]$, either $a_1 \in \mC(i,c)$ or $a_3 \in \mC(i,c)$. By examining the possible cases for $g_1$, whose defining columns must be adjacent, we write down all possible cases for $g_1$ and $g_2$ and show that the $S$-polynomial reduces to zero.

\begin{center}
\begin{tabular}{ccc}
\toprule
$g_1$ & $g_2$ & Reduction of $S(g_1,g_2)$\\
\midrule 
$[12|12]$ & $[134]$ & $- y_1 [234] + y_4z_3 [12|12] + x_3y_4 [23|12] - x_4y_3 [23|12]$\\
$[13|12]$ & $[134]$ & $-z_1[234] + y_4z_3 [12|12]+ x_3z_4 [23|12] - x_4y_3 [23|12]$\\
\midrule
$[13|34]$ & $[124]$ & $- x_4 [123] + x_1z_2 [12|34] + x_2y_1 [13|34] - x_2z_1 [12|34]$ \\
$[23|34]$ & $[124]$ & $-y_4[123] - x_2y_1 [23|34] + y_1z_2 [12|34] - y_2z_1 [12|34]$ \\
\bottomrule
\end{tabular}
\end{center}
The first two cases are for $a_1 \in \mC(i,c)$ and the last two cases are for $a_3 \in \mC(i,c)$.

\smallskip
\textbf{Case 4.ii.} Assume that $P$ contains at least $4$ rows. Since the leading terms of $g_1$ and $g_2$ are not relatively prime, we have that $P$ contains exactly $4$ rows. Let us write, $g_2 = [b_1 b_2 b_3 \ | \ a_1  a_2  a_3]$ where either $a_1 \in \mC(i,c)$ or $a_3 \in \mC(i,c)$. If $a_3 \in \mC(i,c)$ then either $g_1 = [b' b_3 | a' a_3]$ or $g_1 = [b_3  b'' | a_3 a'']$ where $b' < b_3 < b''$ and $a' < a_2 < a''$. If $a_1 \in \mC(i,c)$ then we have that either $g_1 = [ b'  b_1 |  a'  a_1]$ or $g_1 = [b_1  b'' | a_1  a'']$ where $b' < b_1 < b''$ and $a' < a_1 < a''$. For each case we write down the reduction of the $S$-polynomial after a suitable relabelling of the columns of $P$, which preserves their order.

\begin{center}
    \resizebox{\textwidth}{!}{
    \begin{tabular}{ccl}
    \toprule
        $g_1$ & $g_2$ & \multicolumn{1}{c}{Reduction of $S(g_1, g_2)$} \\
    \midrule 
        $[14|34]$ & $[234|124]$ & 
        $-x_{1,4} [234|123]
        +[24|12][13|34]
        -[34|12][12|34]
        -x_{2,2} x_{3,1} [14|34]$
        \\ 
        $[24|34]$ & $[134|124]$ &
        $-x_{2,4}[134|123]
        +[14|12] [23|34]
        +[34|12] [12|34]
        +x_{1,2} x_{3,1} [24|34]$
        \\ 
        $[34|34]$ & $[124|124]$ & 
        $-x_{3,4} [124|123]
        -[14|12] [23|34]
        +[24|12] [13|34]
        +x_{1,2} x_{2,1} [34|34]$
        \\
        $[34|34]$ & $[123|123]$ & 
        $-x_{3,4}[124|123]
        +x_{2,3}[134|123]
        -x_{1,3}[234|124]
        +[34|12][12|34]
        +x_{1,2}x_{2,1}[34|34]$
        \\ 
    \midrule
        $[12|12]$ & $[134|134]$ & 
        $-x_{1,2}[234|134]
        -x_{3,1}[124|234]
        +x_{4,1}[123|234]
        +[12|34][34|12]
        +x_{3,4}x_{4,3}[12|12]$
        \\ 
        $[13|12]$ & $[124|134]$ & 
        $-x_{3,1}[124|234]
        +[14|34] [23|12]
        +[12|34] [34|12]
        +x_{2,4}x_{4,3}[13|12]$
        \\ 
        $[14|12]$ & $[123|134]$ & 
        $-x_{4,1} [123|234]
        +[13|34] [24|12]
        -[12|34] [34|12]
        +x_{2,4}x_{3,3}[14|12]$
        \\
        $[12|12]$ & $[234|234]$ & 
        $-x_{1,2}[234|134]
        +[24|34][13|12]
        -[23|34][14|12]
        +x_{3,4}x_{4,3}[12|12]$
        \\
    \bottomrule
    \end{tabular}
    }
\end{center}

\medskip

\textbf{Case 5.} Let $\deg(g_1) = 2$ and $\deg(g_2) = 4$. Note that the $4$-minors in the generating set are defined on columns $[k\ell] \backslash (S \cup \mC(i,c))$, whereas the $2$-minors on columns $\mC(i,c)$. In particular, $g_1$ and $g_2$ are defined on disjoint sets of variables so their initial terms are relatively prime.

\medskip

\textbf{Case 6.} Let $\deg(g_1) = 4$ and $\deg(g_2) = 3$.
Let us write $g_1 = [R_1 | C_1]$ and $g_2 = [R_2 | C_2]$. By assumption the leading terms of $g_1$ and $g_2$ are not relatively prime, thus $R_1 \cap R_2 \neq \emptyset$ and $C_1 \cap C_2 \neq \emptyset$. Hence, after a relabelling of the rows and columns that preserves their relative order, we may assume that $R_1, C_1, R_2, C_2$ are subsets of $\{1, \dots, 6 \}$. Let $G$ be the collection of:
\begin{itemize}
    \item All $4$-minors of $P$,
    \item All $3$-minors of $P$ on columns $\conv(C_2) := \{\min(C_2), \min(C_2)+1, \dots, \max(C_2)\}$.
\end{itemize}
By definition we have that the ideal $I_S$ contains all $4$-minors of $P$. Also we have that either $C_2 \subseteq \{1, \dots, i+c \}$ or $C_2 \subseteq \{i+1, \dots, i+c+j \}$. Therefore $ \conv(C_2) \subseteq \{1, \dots, i+c \} $ or $ \conv(C_2) \subseteq \{i+1, \dots, i+c+j \}$. Hence, all $3$-subsets of $\conv(C_2)$ lie in  $F(i,j,c)$ and so all $3$-minors of $P$ on columns $\conv(C_2)$ are contained in $I_S$.

We computationally verify that, for each possible $R_1, C_1, R_2, C_2$, the polynomials in $G$ are sufficient to reduce the $S$-polynomial $S(g_1,g_2)$ to zero. The code can be found on GitHub:
\begin{center}
\url{https://github.com/ollieclarke8787/hypergraph_ideals}.\footnote{The code took 11 hours to terminate on a laptop with an Intel i5-8350U CPU @ 1.70GHz with 8GB of memory.}
\end{center}
This completes Buchberger's Criterion for $I_{F(i,j,c)}$. 

\medskip

To show that the generating set for $I_0$ forms a Gr\"obner basis, take two generators $g_1$ and $g_2$ and take cases on their degree. If they have the same degree then the $S$-polynomial reduces to zero by a classical result. On the other hand, if their degrees are different then we may assume $\deg(g_1) = 2$ and $\deg(g_2) = 3$. We may also assume that their initial terms are not relatively prime. The minors can be seen to be taken from a submatrix of $X$ with $4$ rows and columns which we will label $1,2,3,4$, such that the order is compatible with $X$. By the construction of $I_0$, we deduce that $g_1$ is a $2$-minor on adjacent columns of $X$. If $g_1$ is contained in either columns $1,2$ or $3,4$ then by case $4$ of the above argument, the $S$-polynomial reduces to zero. The remaining cases are when the $2$-minor is taken from columns $2,3$, as in the table below. This completes the proof for $I_0$.

\begin{center}
\resizebox{\textwidth}{!}{
\begin{tabular}{ccl}
\toprule
$g_1$ & $g_2$ & 
\multicolumn{1}{c}{Reduction of $S(g_1, g_2)$} \\
\midrule
$[23|23]$ & $[123|124]$ & 
$-x_{3,2}[123|134] 
+[23|14][13|23]
+x_{1,4}x_{3,1}[23|23]$
\\
$[12|23]$ & $[123|134]$ & 
$-x_{1,3}[123|124] 
+[12|14][13|23]
+x_{1,4}x_{3,1}[12|23]$
\\
\midrule
$[23|23]$ & $[134|134]$ & 
$-x_{2,3}[134|124]
-x_{4,2}[123|134]
+x_{4,3}[123|124]
$
\\
& & \multicolumn{1}{r}{
$+[12|14][34|23]
+[23|14][14|23]
-[34|14][12|23]
+x_{1,4}x_{4,1}[23|23]$}
\\
\midrule
$[13|23]$ & $[234|134]$ & 
$-x_{1,3}[234|124]
+[23|14][14|23]
+[34|14][12|23]
+x_{2,4}x_{4,1}[13|23]$
\\
$[23|23]$ & $[124|124]$ & 
$-x_{2,3}[134|124]
-x_{4,2}[123|134]
+x_{1,2}[234|134]
+[23|14][14|23]
+x_{1,4}x_{4,1}[23|23]$
\\
$[24|23]$ & $[123|124]$ &
$-x_{4,2}[123|134]
+[12|14][34|23]
+[23|14][14|23]
+x_{1,4}x_{3,1}[24|23]$
\\
\bottomrule
\end{tabular}
}
\end{center}

Finally to show that the generating set for $I_S$ forms a Gr\"obner basis for $S \in D$, we note that $\Delta_S$ has at least two connected components and so $I_S$ is the sum of the hypergraph ideals associated to each component. For each connected component of $\Delta_S$, the corresponding ideal has a generating set which forms a Gr\"obner basis by the same $S$-polynomial reductions as $I_0$. Since the generating sets of the ideals of each of the components are on disjoint sets of variables, it follows that the generating set for $I_S$ forms a Gr\"obner basis. This completes the proof.
\end{proof}

Using the Gr\"obner basis for the ideals $I_S$ immediately shows that the ideals are radical.

\begin{prop}\label{prop:I_S_radical}
For all $S \subseteq [k\ell]$ the ideal $I_S$ is radical.
\end{prop}

\begin{proof}
By Lemma~\ref{lem:k=2_grobner}, the generating set of $I_S$ for all subsets $S \subseteq [k\ell]$ is a Gr\"obner basis with respect to some term order. The generators are all minors of the matrix of variables so, in particular, they are square-free. And so it follows that $I_S$ is radical.
\end{proof}

\subsection{Minimal components}\label{sec:k=2_min_comp}
We now determine the minimal primes in the decomposition of $\sqrt{I_{\Delta }}$ and prove Theorem~\ref{thm:k=2_min_prime_decomp}.

\medskip

\noindent\textbf{Notation.} For the case $k = 2$, we recall that the matrix $\mY$ has columns $C_1 = \{1,2 \}, C_2 = \{3,4 \}, \dots, C_\ell = \{ 2\ell-1, 2\ell\}$. 
Let $S \subseteq [k \ell]$, we write $Col(S) \subseteq [\ell]$ for the columns that meet $S$. Formally $Col(S) = \{ j \in [\ell] : C_j \cap S \neq \emptyset \}$.

\begin{definition}\label{def:S_minimal}
We call $S$ \emph{minimal} if it is either empty or all the following three conditions hold.
\vspace{-0.3cm}
\begin{enumerate}
    \item[(1)] $Col(R_1  \backslash S) \not\subseteq Col(R_2\backslash S)$ and $Col(R_2 \backslash S) \not\subseteq Col(R_1 \backslash S)$.
    
    \item[(2)] For all $i \in [\ell]$, $C_i \not\subseteq S$.
    
    \item[(3)] $|R_1 \backslash S| \ge 2$ and $|R_2 \backslash S| \ge 2$.
\end{enumerate}
\end{definition}

We show that a subset $S \subseteq [k\ell]$ is not minimal if and only if the ideal $I_S$ is redundant in the prime decomposition of $\sqrt{I_S}$. More precisely, we prove that:

\begin{lemma}\label{lem:minimal_primes}
There exists $S' \subseteq [k \ell]$ such that $I_{S'} \subsetneq I_S$ if and only if $S$ is not minimal. 
\end{lemma}

\begin{proof}
Suppose $S$ is not minimal. We will construct a subset $S'$ such that $I_{S'} \subsetneq I_S$.  We denote by $\mC$ the union of all columns $C_i$ such that $C_i \cap S = \emptyset$. Note that the $2$-minors in $I_S$ are exactly all $2$-minors from columns $\mC$ of the variable matrix $X$.

Suppose that (2) does not hold for $S$. Then there exists $i \in [\ell]$ such that $C_i = \{2i-1, 2i \} \subseteq S$. In this case, we define $S' = S \backslash 2i$. Consider the generating sets of $I_{S'}$ and $I_S$. Since $2i \not\in I_{S'}$ we have that $I_{S'}$ does not contain any variable from column $2i$ of the variable matrix $X$. Therefore $I_{S'} \neq I_S$. Now take any generator $g$ of $I_{S'}$ which does not lie in the canonical generating set of $I_S$. We see that $g$ is a minor of $X$ containing column $2i$. However $I_S$ contains all variables from column $2i$ and so $I_S$ contains $g$. Therefore $I_{S'} \subsetneq I_S$.

Suppose that (2) holds and (1) does not hold. 
We may assume that $Col(R_1 \backslash S) \subseteq Col(R_2 \backslash S)$. Since (2) holds, it follows that $S \subseteq R_1$. In this case we will show that $I_0 \subsetneq I_S$. Note that $S$ is non-empty and so $I_S$ contains variables, however $I_0$ does not contain any variables and so $I_S \neq I_0$. It suffices to show that all $3$-minors of $X$ are contained in $I_S$. Let $a,b,c$ be columns of $X$.

\begin{itemize}
    \item If $a \in S$ then $[abc] \in I_S$ because $I_S$ contains all variables in column $a$ of $X$.
    
    \item Assume that none of $a,b,c$ lies in $S$. If $a,b$ are contained in $\mC$ then $[abc] \in I_S$ because $I_S$ contains all $2$-minors on columns $a,b$ of $X$, which generate the minor. 
    
    \item Assume that none of $a,b,c$ lies in $S$ and no two of $a,b,c$ lie in $\mC$. Note that $R_1 \subseteq S \cup \mC$, so w.l.o.g. assume that $a,b \in R_2 \backslash \mC$. If $c \in R_2$ then by definition $[abc] \in I_{ \Delta} \subset I_S$. On the other hand, if $c \in (R_1 \backslash S) = (R_2 \cap \mC)$ then by definition of $I_S$ we have that $[abc] \in I_S$.
\end{itemize}

Finally, suppose that (1) and (2) hold but (3) does not hold for $S$. Without loss of generality, assume that $|R_1 \backslash S| \le 1$. By (1) we have that $R_1 \backslash S \not\subseteq R_2 \backslash S$. Therefore $|R_1 \backslash S| = 1$ and $|R_2 \cap S| = 1$. Furthermore, by (1), (2) and permutation of the columns of $\mY$ we may assume that $S = \{1,3,\dots, 2\ell-3, 2\ell \}$. In this case let $S' = \{1,3,\dots,2\ell-5,2\ell\}$. Note that $I_{S'}$ does not contain any variable from column $2\ell-3$ of $X$ so $I_{S'} \neq I_S$. Let us take a generator $g$ of $I_{S'}$ which is not a canonical generator of $I_S$. Then $g$ is a minor of $X$, which contains the column $2\ell-3$. However $I_S$ contains all variables from column $2\ell-3$, which generate the minor. And so $g \in I_S$. And so we have shown that when $S$ is not minimal there exists $S'$ such that $I_{S'} \subsetneq I_S$.

\medskip

For the converse, suppose that $S_1$ and $S_2$ are minimal and $S_1 \subsetneq S_2$. Let $\mC_1$ be the union of all columns $C_i$ with $C_i \cap S_1 = \emptyset$ and similarly define $\mC_2$ for $S_2$. We will show that $I_{S_1} \not\subseteq I_{S_2}$. Assume, without loss of generality, that $2i - 1 \in S_2 \backslash S_1$ for some $i \in [\ell]$. Let $a = 2i \in R_2$. By condition (2) from the definition of minimal subsets, we have $C_i \not\subseteq S_2$. Therefore $a \not\in S_2$ and so $a \not\in S_1$. So we have that $C_i \cap S_1 = \emptyset$ and so $C_i \subseteq \mC_1$. There are now two cases, either $\mC_2 = \emptyset$ or $\mC_2 \neq \emptyset$.

\textbf{Case 1.} Assume that $\mC_2 = \emptyset$. By condition (3) from the definition of minimal subsets, we have $|R_1 \backslash S_2| \ge 2$. Let $b, c \in R_1 \backslash S_2$ be distinct elements. Since $\mC_2 = \emptyset$, therefore $b+1, c+1 \in S_2$. Hence $I_{S_2}$ does not contain $[abc]$. However by construction of $I_S$, we have that $[abc] \in I_{S_1}$.

\textbf{Case 2.} Assume that $\mC_2 \neq \emptyset$. Let $b \in \mC_2 \cap R_2$. Then by definition, the $2$-minors on columns $a,b$ of $X$ are contained in $I_{S_1}$. However these $2$-minors are not contained in $I_{S_2}$ because $a \not\in \mC_2$.

And so in each case we have shown that $I_{S_1} \not\subseteq I_{S_2}$ which completes the proof.
\end{proof}

We now ready to state and prove our main result as follows.

\begin{theorem}\label{thm:k=2_min_prime_decomp}
Let $k = 2$. Then the minimal prime decomposition for $\sqrt{I_\Delta}$ is
$
\sqrt{I_\Delta} = \bigcap I_S
$
where the intersection is taken over all minimal subsets $S \subseteq [k\ell]$.
\end{theorem} 
\begin{proof}
By Theorem~\ref{thm:intersection_thm} we know that $I_\Delta$ is the intersection of ideals of line arrangements. Since all line arrangements contain at most two lines, their ideals are prime by Corollary~\ref{cor:irred_4_or_less_lines}. We may refine the intersection of ideals of line arrangements to include only $I_{L(S)}$ for each $S \subseteq [k\ell]$ by Proposition~\ref{prop:I_S=Ar_D(S)}. By Theorem~\ref{thm:radical_I_S_=_I_L}, we have that $\sqrt{I_S} = I_{L(S)}$ for each subset $S \subseteq [k\ell]$. By Proposition~\ref{prop:I_S_radical}, we have $I_S$ is radical hence $I_S = I_{L(S)}$. The minimal prime components are determined by Lemma~\ref{lem:minimal_primes} that $I_S$ is a minimal if and only if $S$ is a minimal as in Definition~\ref{def:S_minimal}. 
\end{proof}

\section{Application to conditional independence}\label{sec:application_CI_statements}

In the field of algebraic statistics, conditional independence (CI) \cite{Studeny05:Probabilistic_CI_structures} has proved to be a valuable tool. The general setup for such questions is as follows. Given a list of random variables and collection of CI statements, determine which distributions on the random variables simultaneously satisfy all CI conditions and detect any additional CI statements implied by the given CI statements. The data of the CI statements can be encoded in a graphical model \cite{MDLW18:Handbook_of_graphical_models}, where inferences about the distributions can be drawn from the combinatorial properties of the graph. In this paper we are particularly interested when some of the variables are hidden, which means that their marginal distributions are not observed. Our aim is to determine when certain constraints on the observed variables arise from conditions on the hidden variables \cite{Steudel-Ay}. We proceed algebraically by noting that distributions satisfying CI statements are precisely the distributions satisfying certain polynomial equations \cite{DrtonSturmfelsSullivant09:Algebraic_Statistics, Sullivant}. The collection of all such polynomials is called a \textit{CI ideal} and has been widely studied \cite{Fink,herzog2010binomial,Rauh,SwansonTaylor11:Minimial_Primes_of_CI_Ideals}. The distributions satisfying the given collection of CI statements can be recovered by intersecting the CI ideal with the probability simplex. In the presence of hidden variables, these polynomial equations become far more complicated and very few cases have been computed and studied \cite{andreas,clarke2020conditional}. More precisely, the classical CI ideals are generated by $2$-minors of a matrix probabilities whereas the CI ideals for models with hidden variables may contain larger minors and thus motivating the definition of determinantal hypergraph ideals.

\subsection{CI statements and ideals}
Let $X, Y, Z$ be discrete random variables with finite range. We say $X$ and $Y$ are \emph{independent}, and write $\ind X Y$, if $P(X = x, Y = y) = P(X = x)P(Y = y)$ for all $x, y$. In other words the joint distribution of $X$ and $Y$ can be factored as the product of the marginal distributions of $X$ and $Y$. We may view this in terms of the joint probability matrix $P_{XY} = (p_{x, y})$, where $p_{x, y} = P(X = x, Y = y)$. The marginal distributions for $X$ and $Y$ are the column vectors $P_X = (p_x)$ and $P_Y = (p_y)$ where $p_x = \sum_y p_{x, y}$ and $p_y = \sum_x p_{x,y}$. The variables $X$ and $Y$ are independent if and only if $P_{XY} = P_X \cdot P_Y^T$, or, as we have seen before, $P_{XY}$ is of rank one. Therefore two random variables are independent if and only if all $2$-minors of their joint probability matrix are zero. If we consider $p_{x,y}$ as variables of a ring $\CC[p_{x,y}]$, then the ideal generated by all $2$-minors is called the \emph{conditional independence ideal $I_{\ind X Y}$}. The probability distributions that appear in the variety are all possible distributions that satisfy $\ind X Y$.

We say that $X$ and $Y$ are conditionally independent given $Z$, and write $\ind X Y | Z$, if 
\begin{equation}\label{eqn:cond_indep}
P(X = x, Y = y | Z = z) = P(X = x | Z = z) P(Y = y | Z = z)\quad\text{for all $x, y, z$}.
\end{equation}
We may state this condition in terms of the joint probability tensor $P_{XYZ} = (p_{x,y,z})$ where $p_{x,y,z} = P(X = x, Y = y, Z = z)$. We imagine $P_{XYZ}$ as a $3$-dimensional tensor. Condition~\eqref{eqn:cond_indep} is equivalent to each $z$-slice $(p_{x,y,z})_{x,y}$ having rank one. The conditional independence ideal $I_{\ind X Y | Z} \subseteq \CC[p_{x,y,z}]$ is generated by the $2$-minors of each $z$-slice. Explicitly each generator of the ideal is of the form $$p_{x_1, y_2, z}p_{x_2, y_2, z} - p_{x_1, y_2, z} p_{x_2, y_1, z}.$$

We say that a variable $Z$ is \emph{hidden} if the marginal distribution is not observed. So if $X, Y$ are observed variables, i.e.~not hidden, then the probabilities we observe are $P_{XY}$. Suppose that the state space of $Z$ has size $k$. Since each $z$-slice of the tensor $P_{XYZ}$ has rank one, it follows that the flattening $P_{XY} = (p_{x,y}) = \left( \sum_z p_{x,y,z} \right)_{x,y}$ has rank at most $k$, hence the $k + 1$ minors of $P_{X,Y}$ vanish. So we define the conditional independence ideal with hidden variables to be $I_{\ind X Y | C} \subseteq \CC[p_{x,y}]$ to be the ideal generated by all $k + 1$ minors of the matrix $(p_{x,y})$ of observed variables. Given a collection $\mC$ of CI statements with hidden variables, we define the conditional independence ideal $I_\mC = \sum I_{\ind X Y \mid Z}$ to be the sum of the conditional independence ideals for each $\ind X Y \mid Z \in \mC$.

\begin{example}\label{example:CI-4-4-2}
{\rm 
Suppose $X, Y$ are observed random variables with state space $S = \{1,2,3,4\}$ and suppose $H$ is a binary hidden random variable which takes values in $\{0, 1\}$. Suppose these variables satisfy the CI statement $\ind X Y \mid H$. The probabilities of the joint distribution of $X,Y$ and $H$ is represented by the $3$-dimensional tensor $P_{XYH} = (p_{x,y,h})$ where $x,y \in S$ and $h \in \{0,1\}$. The CI statement implies that each $H$-slice of $P_{XYH}$ has rank one. Explicitly, these slices are $(p_{x,y,0})$ and $(p_{x,y,1})$.
The observed probabilities are obtained by flattening $P_{XY} = (p_{x,y}) = (p_{x,y,0} + p_{x,y,1})$. It follows that $P_{XY}$ has rank at most two and so all the $3$-minors of $P_{XY}$ are zero. The associated CI ideal lies in the polynomial ring $\CC[p_{x,y} : x,y \in S]$ and is generated by all $3$-minors of $P_{XY}$.
}
\begin{figure}
    \centering
    \includegraphics[scale=0.7]{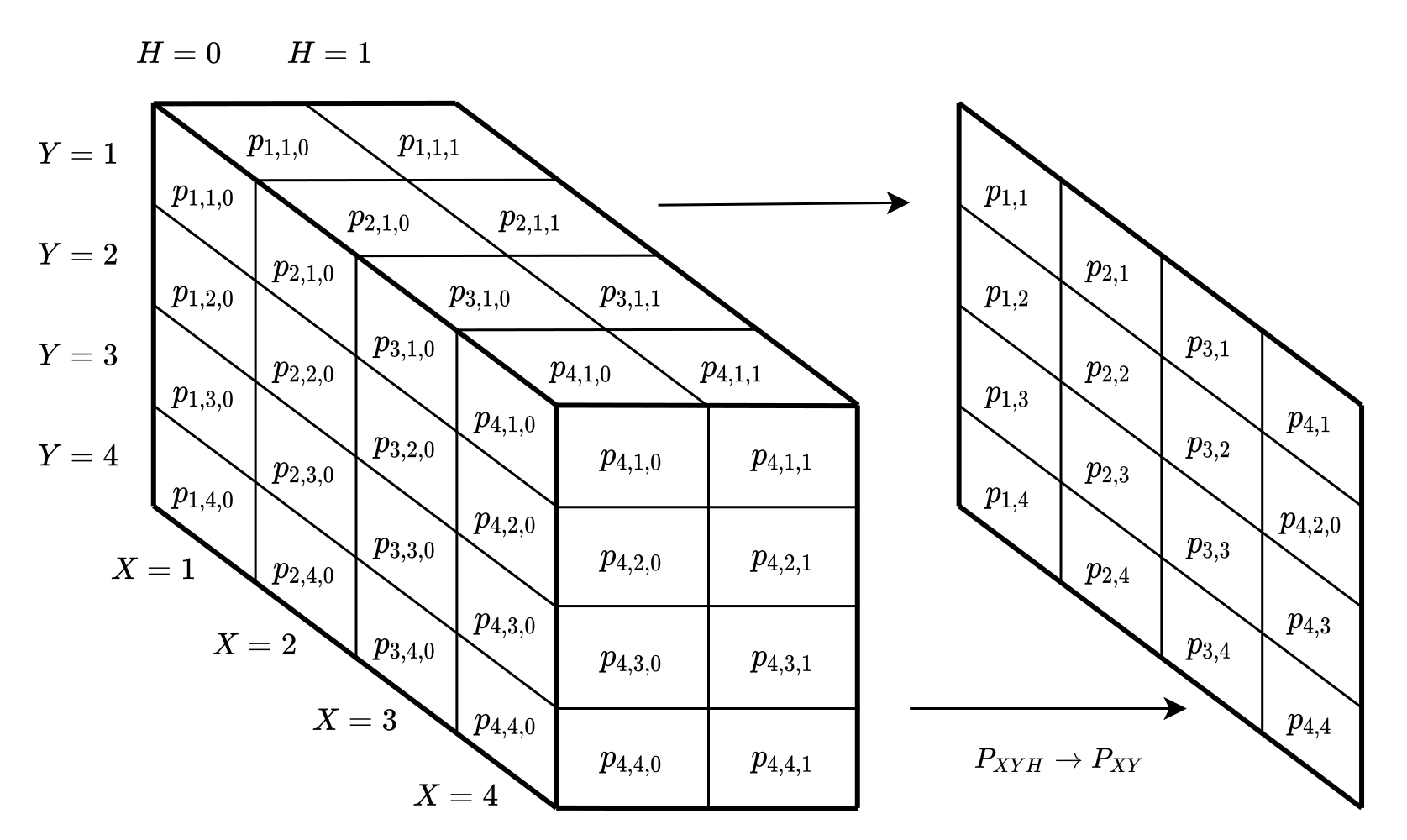}
    \caption{Depiction of the flattening of the tensor $P_{XYH}$ of probabilities from Example~\ref{example:CI-4-4-2}. Each $H$-slice of $P_{XYH}$ has rank one and so the rank of $P_{XY}$ is at most two.}
    \label{fig:flatten_tensor}
\end{figure}
\end{example}

\subsection{Intersection axiom}\label{sec:CI_intersection}

Suppose $X, Y_1, Y_2$ are discrete random variables with finite ranges $\MX, \MY_1, \MY_2$ of size $d, k, \ell$ respectively. Let $H$ be a hidden random variable and consider the following CI model given by:
\begin{eqnarray}\label{eq:C}
\mC : 
\ind X {Y_1} \mid Y_2 \quad \textrm{and} \quad  \ind X {Y_2} \mid \{Y_1, H\}.
\end{eqnarray}
If $H$ is a constant then the \textit{intersection axiom} states that $\ind X {\{Y_1, Y_2\}}$ holds whenever the distributions have strictly positive probabilities. From the above CI statements we can pass to the CI ideal. The ideal $I_\mC$ is the intersection of prime binomial (toric) ideals \cite{Fink} and one of these components is precisely the CI ideal $I_{\ind{X}{\{Y_1, Y_2\}}}$. In terms of varieties, this is the component which contains all distributions that are \textit{fully supported}, i.e.~do not have any structural zeros. All other components contain structural zeros which are well-understood \cite{herzog2010binomial, Rauh}.

\begin{remark}\label{rem:CI}
We consider the ideal $I_\mC$ of the CI model in \eqref{eq:C} when $H$ takes two different values. 
We write $Y$ for the joint distribution of $Y_1$ and $Y_2$. We identify the state space of $Y$ with $[k\ell]$ and arrange this set in a $k \times \ell$ matrix $(\MY_{i,j})$.
The hypergraph ideal $I_\Delta \subseteq R$ studied throughout this paper coincides with the conditional independence ideal of $\MC$. The state space of the joint distribution of observed variables is in bijection with the variables of the ring $R$. For notation purposes, let us write $P = (p_{i,j})$ for the matrix of variables in $R$, where $i \in [d]$ and $j \in [k\ell]$.

The construction of $I_\Delta$ can be seen explicitly as follows. The ideal $I_{\ind X {Y_1} \mid Y_2}$ is generated by all $2$-minors of $Y_2$-slices of the matrix $P$. Each $Y_2$-slice corresponds to a column of the matrix $\MY$. So, for each column $C_i$ of $\MY$, the ideal $I_{\ind X {Y_1} \mid Y_2}$ contains all $2$-minors of submatrices of $P$ with columns indexed by $C_i$.
The ideal $I_{\ind X {Y_2} \mid \{Y_1, H \}}$ is generated by all $3$-minors of $Y_1$-slices of $P$. Each $Y_1$-slice corresponds to a row of the matrix $\MY$. So for each row $R_j$ of $\MY$, the ideal $I_{\ind X \{Y_2\} \mid \{Y_1, H\}}$ contains all $3$-minors of submatrices of $P$ with columns indexed by $R_j$. 
\end{remark}

\subsection{Interpretations and results}

Understanding the irreducible decomposition of the variety of the CI ideal allows us to explore the distributions that satisfy the given CI statements.

\begin{remark}
Theorem~\ref{thm:intersection_thm} allows us to write $V(I_\Delta)$ as a union of configuration spaces of line arrangements. The arrangements $Ar_{\Delta}(\emptyset)$ have particular statistical significance since the distributions which lie inside do not contain any structural zeros \cite{Sturmfels02:Solving_polynomial_equations}.
For $k = 2$, we have shown that $I_0$ is the unique prime component of $\sqrt{I_{\Delta}}$ which does not contain any monomials. It follows that $I_0 = (\sqrt{I_\Delta} : \prod p_{x,y} ^\infty)$ is given by saturating the ideal with respect to the product of variables in $R$. The ideal $I_0$ is the CI ideal associated to $\ind X {\{Y_1, Y_2\}} \mid H$. We may therefore deduce a hidden variable version of the intersection axiom as follows:
\[
\MC = \{\ind X {Y_1} \mid Y_2, \ \ind X {Y_2} \mid \{Y_1, H\} \}
\implies
\ind{X}{\{Y_1, Y_2\} \mid H}.
\]
All other prime components are of the form $I_S$ for some subset $S$. The monomials in $I_S$ correspond exactly to $S$, i.e.~the zero columns of a generic point $A \in V(I_S)$ are those indexed by $S$.
\end{remark}

\begin{remark}
There are many computational resources which aid in the calculation of binomial ideals that arise from CI statements without hidden random variables. In general these ideals and their prime components are all binomial. 
However, once hidden variables are introduced, computations quickly become intractable. Even in the presence of a binary hidden random variable $H$, the current built-in algorithms in computer algebra systems such as {\tt Singular} and {\tt Macaulay} cannot handle the case $k=3, d=3, \ell=4$. 
The decomposition of the associated ideal leads to 319 components: one of them corresponds to the CI statement $\ind{X}{\{Y_1,Y_2\}}\mid H$ which implies the ``intersection axiom'' and the rest of the components represent probability distributions which lie on the boundary of the probability simplex, i.e.~they have structural zeros. These components correspond to nine types of line arrangements which have at most three lines, see Example~\ref{example:k3l4s2td3}.

If we restrict our attention to $k = 2, d = 3$ then $\ell = 5$ is the smallest $\ell$ for which the prime decomposition of $I_\Delta$ cannot be determined directly by current programs. In this case we have seen that there are 171 prime components which fall into six isomorphism classes, see Example~\ref{example:k2l5s2td3}. Each prime component corresponds to a distinct line arrangement with at most two lines.
\end{remark}

\section{Further questions}\label{sec:further_qus}
We now consider how the results in this paper may generalise and provide calculations to support these claims. We begin by summarising the largest example of a minimal prime decomposition of $\ID$ that we have been able to verify by computer. In this example we have $k = 3$ and so we conjecture that the prime components correspond to point and line arrangements with at most three lines. We then explore the extent to which the Gr\"obner basis results 
may generalise. 

\begin{example}
\label{example:k3l4s2td3}
{\rm Let $k = 3, \ell = 4$, so we have 
$\mY = 
\begin{bmatrix}
1   &4  &7  &10 \\
2   &5  &8  &11 \\
3   &6  &9  &12 \\
\end{bmatrix}
$ and
\vspace{-0.2cm}
\begin{equation*}
\resizebox{\textwidth}{!}{
$\Delta = \left\{\binom{\{1,2,3\}}{2},\binom{\{4,5,6\}}{2},
\binom{\{7,8,9\}}{2}, \binom{\{10,11,12\}}{2},
\binom{\{1,4,7,10\}}{3}, \binom{\{2,5,8,11\}}{3},
\binom{\{3,6,9,12\}}{3} \right\}.$}
\end{equation*}
If $d=3$ then we are able to verify the prime decomposition of $I_{ \Delta}$. We list the prime components below, up to isomorphism. Note that each component is a hypergraph ideal whose hypergraph can be uniquely identified by its singletons. If $S$ is the set of singletons of the hypergraph, we write $I_S$ for its associated ideal.
}
{\rm We note that the ideals in the last $5$ rows of the table below are generated by variables and $2$-minors. These ideals are closely related to generalised binomial edge ideals, see \cite{Rauh}. Since the connected components of the hypergraphs are cliques, it follows that these ideals are all radical and prime. Geometrically the components in the last $5$ rows correspond to arrangements of points. For instance, the ideal $I_{1,4,5,8,9,12}$ corresponds to the arrangement of points $\{2,3,6,7,10,11 \}$ where the pairs of points $2,3$ and $10,11$ coincide.  

The other components, appearing at the top of the table, are ideals of line arrangements. The ideal $I_0$ corresponds to a line arrangement with a single line and all points lying on this line. The other ideals also correspond to line arrangements that have exactly one line, however in each case there is a point lying off the line. For larger values of $\ell$, we expect that prime components of $I_\Delta $ correspond to line arrangements with at most three lines and are no longer determinantal. For instance, we would expect to see prime components with three lines meeting at a point, see Example~\ref{example:three_lines}.
A graphical representation of these components is given in Figure~\ref{fig:example1components}
}
\begin{center}
\resizebox{0.85\textwidth}{!}{
\begin{tabular}{ccc}
    \toprule
   {\rm  Type of Ideal}  &
    {\rm Hypergraph of ideal} & 
    \begin{tabular}{c} {\rm Number of ideals} \\ {\rm with the same type}
    \end{tabular}  \\
    \midrule
    $I_0$               &
    $\left\{\binom{C_i}{2}  \textrm{ for all } i \in [4], \ \binom{[12]}{3} \right\} $ & 
    $1$    \\
    $I_{1,2,6,9}$       &
    $\left\{1,2,6,9,\{4,5 \}, \{7,8 \}, \binom{\{10,11,12 \}}{2}, \binom{\{4,5,7,8,10,11,12\}}{3} \right\}$
    & 
    $36$                    \\
    $I_{1,5,6,8,9}$     & 
    $\left\{1,5,6,8,9,
    \{2,3 \}, \binom{\{ 10, 11, 12\}}{2},  \binom{\{4, 7, 10, 11, 12\}}{3}\right\} $
    &
    $36$                    \\
    $I_{1,5,6,8,12}$    & 
    $ \left\{
    1,5,6,8,12,
    \{2,3\},
    \{7,9 \},
    \{10,11 \},
    \binom{\{4,7,9,10,11\}}{3}
    \right\}$
    & 
    $72$                    \\
    \midrule
    $I_{1,4,8,12}$    &
    $\left\{1, 4, 8, 12, \binom{\{2,3,5,6\}}{2}, \{7,9\} , \{10,11 \} \right\}$ &
    $24$                    \\
    $I_{1,2,6}$       & 
    $\left\{1, 2, 6, \{4,5\}, \binom{\{7,8,9,10,11,12\}}{2} \right\}$ &
    $36$                    \\
    $I_{1,2,7,9,11,12}$ & 
    $\left\{1,2,7,9,11,12, \binom{\{4,5,6\}}{2} \right\}$ &
    $24$                  \\
    $I_{1,2,4,5,9,12}$  & 
    $\left\{1,2,4,5,9,12, \{7,8 \},\{10,11\} \right\}$ &
    $18$                 \\
    $I_{1,4,5,8,9,12}$  & 
    $\left\{1,4,5,8,9,12, \{2,3 \},\{10,11\} \right\}$ &
    $72$                  \\
    \bottomrule
\end{tabular}
}
\end{center}
\begin{figure}
    \centering
    \resizebox{0.9\textwidth}{!}{
    \includegraphics{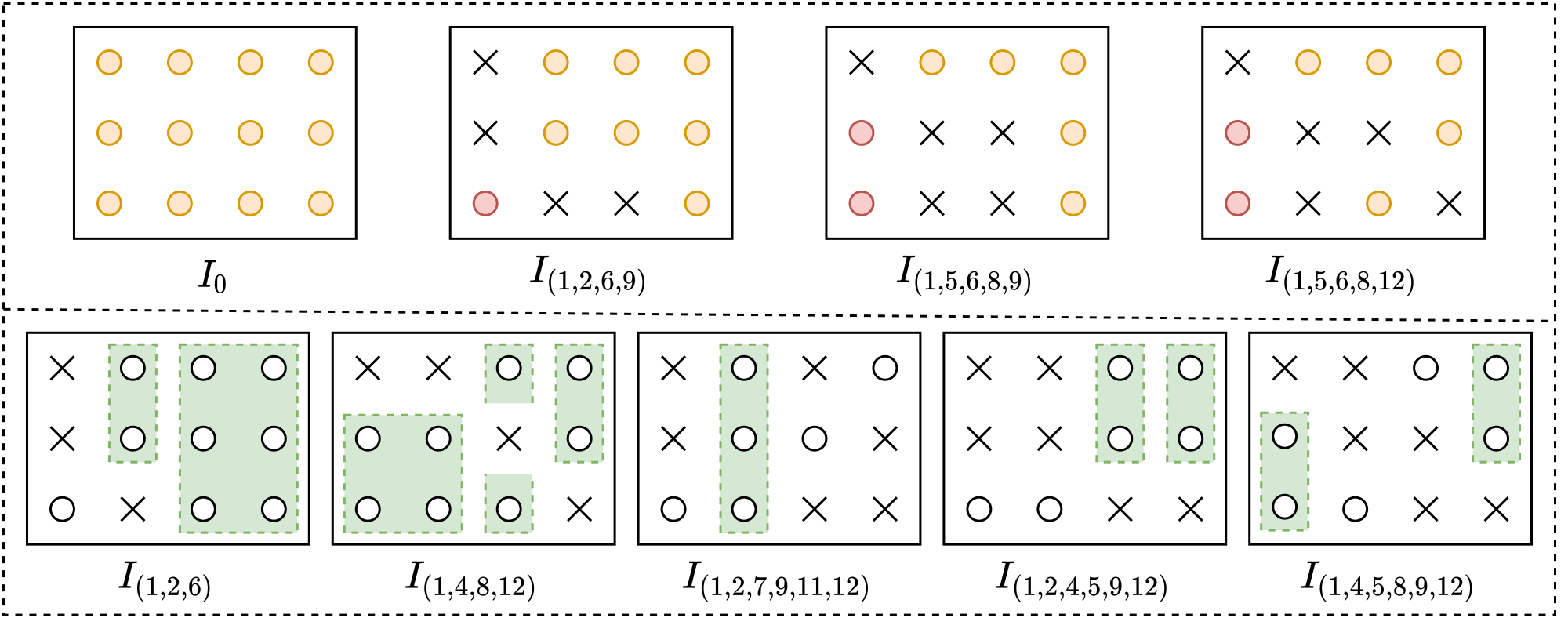}}
    \caption{A diagram depicting the hypergraphs for each prime component of $I_{ \Delta}$ in Example~\ref{example:k3l4s2td3}. The crosses indicate the singleton edges. The ideals in the lower part of the diagram are those generated by $2$-minors and variables. The shaded regions depict the cliques of the hypergraph. In the top row of the diagram, the shaded circles indicate connected components of $\Delta_S$.}
    \label{fig:example1components}
\end{figure}
\end{example}

In Example~\ref{example:k3l4s2td3}, each hypergraph ideal corresponds to a line arrangement with at most one line.
Hence we can apply our results to study these components. And so it is natural to ask:
\begin{question}
Are the minimal prime components of $\ID$ in general the ideals of line arrangements?
\end{question}

\medskip

We have seen that the non-determinantal generators of $I_L$ for small line arrangements $L$ coincide with the condition in Example~\ref{example:three_lines}. 
Let us define the following collection of polynomials.

\begin{definition}
\label{def:gen_set_for_IL}
Let $L = (\mP, \mL, \mI)$ be a line arrangement on $[n]$ and fix $d \ge 3$. Recall that $X$ is a $d \times n$ matrix of variables.  We let $G_L$ be the following collection of polynomials in $\CC[X]$,
\begin{itemize}
    \item $[A | B]$, where $A \subseteq [d] $, $B \subseteq P$, $P \in \mP$ and $|A| = |B| = 2$.
    \item $[A | B]$, where $A \subseteq [d] $, $B \subseteq \bigcup_{P \in L_i}P$, $L_i \in \mL$ and $|A| = |B| = 3$.
    \item $[cda][efb] - [cdb] [efa]$,
    taken over all points such that the pairs $(a,b), (c,d), (e,f)$ lie on three distinct lines that are incident to a common point.
\end{itemize}
\end{definition}

\begin{question}
For which line arrangements $L$ do we have $I_L = \langle G_L \rangle$? And for which line arrangements does $G_L$ form a Gr\"obner basis?
\end{question}

We note that for all cases of $G_L$ that we have been able to successfully compute, there always appears to be an ordering of the columns of variable matrix $X$ such that the generating set $G_L$ is a Gr\"obner basis with respect to the standard lex order.

\begin{conjecture}\label{conj:_I_L_gen_G_L}
If $L$ is a line arrangement with at most $4$ lines then $G_L$ generates $I_L$.
\end{conjecture}

We have seen that when $k=2$, for any line arrangement $L$, the generators of $G_L$ form a Gr\"obner basis for $I_L$ and the space of configurations is given by configurations of at most two lines. 
We have verified Conjecture~\ref{conj:_I_L_gen_G_L} for the line arrangements consisting of three lines meeting at a point where one line contains exactly three points. We are able to use Gr\"obner bases to show that $\langle G_L \rangle$ is radical. The non-determinantal generators that appear in $G_L$ can be viewed as the remainder of a certain $S$-polynomial after applying the division algorithm on the determinants in $G_L$.

\medskip

\noindent{\bf Acknowledgement.} 
F. Mohammadi and H.J. Motwani were partially supported by BOF grant 
STA/201909/038, EPSRC Early-Career Fellowship EP/R023379/1, and 
FWO grants (G023721N, G0F5921N).
O. Clarke is supported by EPSRC Doctoral Training Partnership 
award EP/N509619/1.

\smallskip
\bibliographystyle{unsrt} 
\bibliography{Det.bib}

\bigskip
\noindent
\footnotesize {\bf Authors' addresses:}

\bigskip 

\noindent Department of Mathematics, University of Bristol, Bristol, UK \\
E-mail address: {\tt oliver.clarke@bristol.ac.uk}

\medskip

\noindent Department of Mathematics: Algebra and Geometry, Ghent University, 9000 Gent, Belgium \\
Department of Mathematics and Statistics,
UiT – The Arctic University of Norway, 9037 Troms\o, Norway
\\ E-mail address: {\tt fatemeh.mohammadi@ugent.be}

\medskip

\noindent Department of Mathematics: Algebra and Geometry, Ghent University, 9000 Gent, Belgium 
\\ E-mail address: {\tt harshitjitendra.motwani@ugent.be}

\end{document}